\documentclass[12pt,a4paper]{amsart}
\usepackage[latin1]{inputenc}
\usepackage[english]{babel}
\usepackage{amsmath,amstext,amsthm,amsfonts,amssymb,amscd}
\usepackage[dvips]{graphicx}
\usepackage{hyperref}

\usepackage[pagewise]{lineno}

\usepackage{cleveref}

\usepackage{color}

\theoremstyle{plain}
\newtheorem{theorem}{Theorem}[section]

\newtheorem{claim}{Claim}[section]

\newtheorem{corollary}{Corollary}[section]
\newtheorem{definition}{Definition}[section]
\newtheorem{example}{Example}[section]

\newtheorem{lemma}{Lemma}[section]

\newtheorem{proposition}{Proposition}[section]
\newtheorem{remark}{Remark}[section]

\newtheorem{theorema}{Theorem}[subsection]

\newtheorem{question}{Question}

\newcommand{\bc}{\begin{center}}
	\newcommand{\ec}{\end{center}}

\def\R{{\mathbb{R}}}

\def\||{\parallel}

\begin{document}
	
	\title[Uniqueness of Equil. Measure for a Family of Part. Hyperb. Horseshoes]{Uniqueness of Equilibrium Measure for a Family of Partially Hyperbolic Horseshoes}

	
	
	\author[K. Oliveira]{Krerley Oliveira}
	\address{Krerley Oliveira, Instituto de Matem\'atica, Universidade Federal de Alagoas, 57072-090 Macei\'o, Brazil} \email{krerley@gmail.com}
	
	\author[M. Oliveira]{Marlon Oliveira}
	\address{Marlon Oliveira, Departamento de Matem\'atica e Inform\'atica, Universidade Estadual do Maranh\~ao,  65055-310  S\~ao Lu\'is, Brazil} \email{marlon.cs.oliveira@gmail.com}
	
	\author[E. Santana]{Eduardo Santana}
	\address{Eduardo Santana, Universidade Federal de Alagoas, 57200-000 Penedo, Brazil}
	\email{jemsmath@gmail.com}
	


	\date{\today}

	
	
	\maketitle
	\begin{abstract}
		In this paper, we show the uniqueness of equilibrium state for a family of partially hyperbolic horseshoes, introduced in \cite{DHRS} for some classes of continuous potentials. For the first class, the method used here is making use of the Sarig's theory for countable shifts. For this purpose, we study the dynamics of an induced map associated to the horseshoe map, we build a symbolic system with infinitely many symbols that is topologically conjugated to this induced map and we show that the induced potential is locally H\"{o}lder and recurrent. For the second class, by following ideas of \cite{RS2} and \cite{RS}, we prove that uniqueness for the horseshoe is equivalent to uniqueness for the restriction to a non-uniformly expanding map, which is a hyperbolic potential and then has uniqueness. Both classes include potentials with high variation, differently from previous results for potentials with low variation. We also prove uniqueness when the potential presents its supremum at a special fixed point and less than the pressure.
	\end{abstract}

	\bigskip


	\section{Introduction}\label{Introduction}
	\vspace{0.50cm}
	
	The theory of equilibrium states on dynamical systems was firstly developed by Sinai, Ruelle and Bowen in the sixties and seventies. 
	It was based on applications of techniques of Statistical Mechanics to smooth dynamics.
	Given a continuous map $f: M \to M$ on a compact metric space $M$ and a continuous potential $\phi : M \to \mathbb{R}$, an  {\textit{equilibrium state}} is an 
	invariant measure that satisfies a variational principle, that is, a measure $\mu$ such that
	$$
	\displaystyle h_{\mu}(f) + \int \phi d\mu = \sup_{\eta \in \mathcal{M}_{f}(M)} \bigg{\{} h_{\eta}(f) + \int \phi d\eta \bigg{\}},
	$$
	where $\mathcal{M}_{f}(M)$ is the set of $f$-invariant probabilities on $M$ and $h_{\eta}(f)$ is the so-called metric entropy of $\eta$.

	In the context of uniform hyperbolicity, which includes uniformly expanding maps, equilibrium states do exist and are unique if the potential is H\"older continuous
	and the map is transitive. In addition, the theory for finite shifts was developed and used to achieve the results for smooth dynamics.
	
	Beyond uniform hyperbolicity, the theory is still far from complete. It was studied by several authors, including Bruin, Keller, Demers, Li, Rivera-Letelier, Iommi and Todd 
	\cite{BK,IT1,IT2,LRL2011,LRL2014} for interval maps; Denker and Urbanski  \cite{DU} for rational maps; Leplaideur, Oliveira and Rios 
	\cite{LOR} for partially hyperbolic horseshoes; Buzzi, Sarig and Yuri \cite{BS,Y}, for countable Markov shifts and for piecewise expanding maps in one and higher dimensions. 
	For local diffeomorphisms with some kind of non-uniform expansion, there are results due to Oliveira \cite{O}; Arbieto, Matheus and Oliveira \cite{AMO};
	Varandas and Viana \cite{VV}. All of whom proved the existence and uniqueness of equilibrium states for potentials with low oscillation. Also, for this type of maps,
	Ramos and Viana  \cite{RV}  proved it for potentials so-called \textit{hyperbolic}, which includes the previous ones. The hyperbolicity of the potential is
	characterized by the fact that the pressure emanates from the hyperbolic region. In all these studies the maps do not have the presence of critical sets and recently, Alves, Oliveira and Santana proved the existence of at most finitely many equilibrium states for hyperbolic potentials, possible with the presence of a critical set (see \cite{AOS}). More recently, Santana completed this by showing uniqueness in \cite{S}. Moreover, in this work it is proved that the class of hyperbolic potentials is equivalent to the class of continuous zooming potential (which satisfies a key inequality between free energies). It includes the null potential, which implies existence and uniqueness of measure of maximal entropy. A similar work is developed in \cite{PV}, where they deal with expanding measures and potentials.
	
	In this work, we deal with equilbrium states in the context of partially hyperbolic dynamics. More precisely, for a special family of partially hyperbolic horseshoes introduced by \textit{Diaz et al} in \cite{DHRS}. For this family, \textit{Leplaideur, Oliveira, Rios} proved in \cite{LOR} that for any continuous potential we have existence of equilibrium states and for a residual subset of continuous potentials they obtained uniqueness. After that, \textit{Arbieto, Prudente} proved uniqueness in \cite{AP} for a class of H\"{o}lder potentials that are constant along the center-stable direction. Later on, \textit{Rios, Siqueira} obtained uniqueness in \cite{RS} for H\"{o}lder potentials with low variation and constant along the unstable direction. Finally, \textit{Ramos, Siqueira} proved the uniqueness of equilibrium state for this horseshoe class and H\"{o}lder potentials with low variation without restrictions in the directions. 
	
	Our results prove uniqueness for classes of continuous potentials with high variation with some restrictions. First, we prove it for what we call \textit{admissible potentials}, where the corresponding potentials for the subshift associated to the horseshoe are locally 
	H\"{o}lder. We create a countable alphabet over the subshift and a full shift to use the Sarig's theory on thermodynamic formalism for countable shifts.  We also prove uniqueness for what we call \textit{projective hyperbolic potentials}, with a condition for the restriction to a non-uniformly expanding map which implies hyperbolicity. We do it by proving that the condition implies that the restriction is an expanding potential and based on a result in 
	\cite{AOS} we have equivalence with hyperbolic potentials. So, by using a result in  \cite{RV} we obtain uniqueness. 

    To obtain all this, we prove that the uniqueness of equilibrium state for the horseshoe is equivalent to the uniqueness for the restriction. Moreover, we also give a result for uniqueness when the potential gives supremum at very important fixed point and it is less than the pressure.
	
	\section{Setup and Main Results}
	
	We emphasize that the question concerning the uniqueness of equilibrium states and phase transitions for partially hyperbolic maps is wide open. In this paper, we make a contribution to this theme by proving the uniqueness of equilibrium states for a very interesting family of partially hyperbolic horseshoes, introduced by \emph{Diaz et al} in \cite{DHRS} and two classes of continuous potentials. We reproduce here the ideas for the construction of this family of horseshoes. Consider the cube $[0,1]\times[0,1]\times[0,1]$, the parallelepipeds:
	
	\[
	R_{0}:=I\times I\times[0,1/6]   \      \      \text{and}     \       \    R_{1}:=I\times I\times[5/6,1]
	\]
	and constants $0<\lambda_{0}<\frac{1}{3}$, $\beta_{0}>6$, $0<\sigma <
	\frac{1}{3}$ and $3<\beta_{1}<4$. Let $f$ be the  time-one map of the vector field $y'=(1-y)y$, i.e.,   $f$ is defined by
	\[
	f^{n}(y)=\frac{1}{1-(1-\frac{1}{y})e^{-n}},
	\]
	for every $n\in\mathbb{Z}$ and $y\neq 0$. Define the family of \textit{horseshoes maps} $F=F_{\lambda_0,\beta_0, \beta_1,\sigma}:R_0\cup R_1 \rightarrow \mathbb{R}^{3}$  by
	\[
	F(x,y,z)=(\lambda_{0}x,f(y),\beta_{0}z),
	\]
	whenever $(x,y,z)\in R_{0}$ and
	\[
	F(x,y,z)=\left(\frac{3}{4}-\lambda_{0}x,\sigma(1-y),\beta_{1}(z-\frac{5}{6})\right),
	\] for $(x,y,z)\in R_{1}$.
	
	Let $\Lambda$ be the maximal invariant set 
	$$	\Lambda=\cap_{n\in\mathbb{Z}}F^{n}(R_0\cup R_1).$$

	The set $\Lambda$ is a partially hyperbolic set  and has an heterodimensional cycle, i.e., the points $Q=(0,0,0)$ and $P=(0,1,0)$ are saddles of $F$ with indices 2 and 1 respectively, and such that
	\[
	W^{s}(P)\cap W^{u}(Q)\neq \emptyset   \    \     \   \text{and}   \      \      \     W^{s}(Q)\cap W^{u}(P)\neq \emptyset.
	\]
	The homoclinic class of $P$, defined by $H(P,F)= \overline{W^{s}(P)\cap W^{u}(P)}$, coincides with $\Lambda$ and contains the homoclinic class of the point $Q$ that is trivial, i.e.,   
	$\Lambda=H(P,F) \supset H(Q,F) = \{Q\}$.
	
	Given $X=(x^s,x^c,x^u)\in\Lambda$, we consider the \emph{central manifold}
	\[
	W^{c}(X)=\left\{(x_{s},y,x_{u}); \ y\in[0,1]\right\},
	\]
	tangent to the central subspace $E^c$ at $X$ and the invariant subset
	\[
	\hat{\Lambda}=\{X\in\Lambda \ ; \ \Lambda\cap W^{c}(X)=\{X\} \},
	\]
	composed by the points whose central manifold is trivial. The set $\hat{\Lambda}^c$ can be decomposed as disjoint uncountable union of these segments and in addition the sets $\hat{\Lambda}$ and $\hat{\Lambda}^c$ are dense on $\Lambda$. We summarize these results in the following theorem:
	\begin{theorem}\label{t.DHRS}(\cite{DHRS}, Theorem 1.)
		Consider $F:R\rightarrow \mathbb{R}^{3}$ the horseshoe map, then
		\begin{enumerate}
			\item The diffeomorphism $F$ has a heterodimensional cycle associated with the saddles $P$ and $Q$. Moreover, there exists
			a non-transverse intersection between $W^{s}(Q)$ and $W^{u}(P)$ whose orbit is contained in $H(P,F)$.
			\item The homoclinic class of $Q$ is trivial and contained in the non-trivial homoclinic class of $P$. In particular, $H(P,F)$
			is not hyperbolic.
			 For $\Sigma_{11}$ the subshift of 
				finite type where only the transition $1 \to 1$ is forbidden, there is a surjection
			\[
			h:\Lambda\rightarrow\Sigma_{11}     \   \   \   \    \    \textrm{with}      \   \   \   \    \    \   h\circ F=\sigma\circ h.
			\]
			Moreover, there exist infinitely many non-trivial central segments $J=\{x\}\times[a,b]\times\{z\}$ contained in the homoclinic class $H(P,F)$ and $h$ is constant on $J$.
		\end{enumerate}
	\end{theorem}
	
	The next lemma gives us important information about the invariant measures $\delta_{Q}$ and $\delta_{P}$.
	
	\begin{lemma}\label{Red. Med.}
		If $\mu$ is an $F$-invariant measure such that $\mu(R_{1})=0$, then $\mu$ is  one of the Dirac measures supported in the points
		$P$ and $Q$ or a convex combination of these measures.
	\end{lemma}
	\begin{proof}
		Given an ergodic $F$-invariant measure $\mu$ such that $\mu(R_{1})=0$, we have supp$\mu \subset R_{0}$, and as the set supp$\mu$ is invariant, we conclude that it must be contained in $\{0\}\times[0,1]\times\{0\}$, since any point in $R_0$ and outside this central curve, eventually leaves  $R_0$ under $F$-iterations.  Thus, using that any point of $\{0\}\times(0,1]\times\{0\}$ is  attracted  by $P$, $\mu$ must be $\delta_Q$ or $\delta_P$. If $\mu$ is an invariant measure such that $\mu(R_{1})=0$, by the Ergodic Decomposition Theorem we conclude  that $\mu$ must to be a convex combination of these two Dirac measures.
	\end{proof}

	The study of equilibrium states for the horseshoe map $F$ started in \cite{LOR}. In this work, the authors proved  that every ergodic invariant measure is \textit{hyperbolic} (in the sense that the Lyapunov exponents are all different from zero) and  there is a gap in the set of the central Lyapunov exponents. In addition, they showed that every recurrent point different from $Q$ and $P$ belongs to $\hat{\Lambda}$. As a consequence of these results, they obtained the existence of equilibrium state for every continuous potentials and uniqueness for a residual set of potentials. Moreover, they presented a family of continuous potentials H\"{o}lder continuous that has a \textit{phase transition}. More precisely, the following statement holds.
	
	\begin{theorem}\label{t.LOR}(\cite{LOR}, Theorem 2.1, 2.2 , Proposition 3.3, Remark 4.)
		
		Let $F$ be partially hyperbolic horseshoe map. Then
		\begin{enumerate}
			\item For any recurrent point $X$ different from $Q$:
			\[
			\liminf_{n\rightarrow +\infty}\frac{1}{n}\log|DF^{n}(X)|_{E^{c}}|\leq0.
			\]
			Moreover, any ergodic invariant measure for $F$ different from $\delta_{Q}$ has negative central Lyapunov exponent.	
			\item Let $X$ be a point in $\Lambda$. If there exists some natural number $k$ such that $h(F^{n}(X))\in[1\!\underbrace{0...0}_{k}\!1]$ for infinitely many $n\in\mathbb{Z}$, then
			\[
			\Lambda\cap W^{c}(X)=\{X\}.
			\]	
			\item Any continuous potential $\phi:\Lambda\rightarrow \mathbb{R}$ admits an equilibrium state,
			and there exists a residual set in $C^{0}(\Lambda)$ such that the equilibrium state is unique.
			
			\item  The family of potentials  $\phi_{t}=t\log|DF|_{E^{c}}|$, for $t\in \mathbb{R}$, has a  \textit{phase transition}, i.e., there exists $t_{0} > 0$ such that $\phi_{t_{0}}$
			admits at least two different equilibrium states.
		\end{enumerate} 	
	\end{theorem}
	
	\begin{remark}\label{Remark1}
		Arbieto and Prudente in \cite{AP} proved that the horseshoe map $F$ has an unique maximal entropy measure $\mu_{max}$. Since $h_{top}(F)=\log \frac{1+\sqrt{5}}{2}>0$, using Lemma \ref{Red. Med.} we obtain that $\mu_{\max}(R_{1})>0$. Moreover, by Theorem \ref{t.LOR} item (2) the set $\hat{\Lambda}$ contains all recurrent points, so we conclude that  $\mu_{max}(\hat{\Lambda})=1$.
	\end{remark}
	
	In \cite{AP}, it was obtained the uniqueness of equilibrium state for the horseshoe map $F$ with respect to H\"{o}lder continuous potentials $\phi:\Lambda\rightarrow \mathbb{R}$ that do not depend on the center-stable directions, that is, such that there exists a H\"{o}lder  continuous potential $\theta$ such that $\phi (X) = \theta (z)$ for every point $X=(x,y,z)\in \Lambda$. The uniqueness also was obtained in \cite{RS}, considering H\"{o}lder  continuous potentials $\phi$ on $\Lambda$ with small variation and that do not depend on the unstable direction, that is, such that there exists a H\"{o}lder  continuous potential $\theta$ such that $\phi (X) = \theta (x,y)$ for every point $X=(x,y,z)\in \Lambda$. The following result in \cite{RS2} extends it without the restriction on the unstable direction.
	\begin{theorem}\label{RS2}
		Let $F: \Lambda \rightarrow \Lambda$ be the horseshoe map and let $\phi:\Lambda \rightarrow  \mathbb{R}$ be a H\"{o}lder  continuous potential that satisfies the condition:
		\begin{equation}\label{smallvariation}
			\sup\phi - \inf \phi < \frac{\log\omega}{2},   \    \     \    \    \   (\textrm{\textit{small variation}})
		\end{equation}
		where $\omega =\frac{1+\sqrt{5}}{2}$. Then there exists a unique equilibrium state for the system $F$ with respect to the potential $\phi$. 
	\end{theorem} 
	
	In \cite{LRL2014}, for rational maps, it was considered a class of potential called hyperbolic that include potentials that satisfies  \eqref{smallvariation} and also some potential with large variation (\cite{LRL2014}, Lemma A.1) . In \cite{RV} it was introduced an another class of hyperbolic potentials that include potentials that satisfy the condition \eqref{smallvariation}. In this way, questions naturally arise:
	\begin{question}
		Are there hyperbolic potentials in the senses of the works \cite{LRL2014} and \cite{RV} with large variation and that admit uniqueness of equilibrium states for the Horseshoe $F$?
	\end{question}
	
	
	Our results provide answers to this question. That is, we obtain hyperbolic potentials in both contexts that admit a unique equilibrium state and with big variation.
	
	In order to state our main results, we need some definitions.

	In the Section~\ref{Construction of I.E}, we construct an inducing scheme using a parametrized family of subsets $\Sigma_\alpha \subset \Sigma$ and a respective inducing time $\rho: \Sigma_\alpha \rightarrow \mathbb{N}$ such that the induced map $T=\sigma^\rho: \Sigma_\alpha \rightarrow \Sigma_\alpha$ is topologically conjugated to a full shift $\sigma: S^\mathbb{Z} \rightarrow S^\mathbb{Z}$ on an countable alphabet $S$, i.e., there exists a continuous homeomorphism $\Pi: S^{\mathbb{Z}} \rightarrow \Sigma_\alpha$ such that $\Pi \circ  \sigma = T \circ \Pi$. Given a potential $\varphi:\Sigma \rightarrow \mathbb{R}$, we define the induced potential as $\varphi_\rho(w) = \sum_{k=0}^{\rho(w)-1} \varphi(\sigma^{k}(w))$ and $\Psi = \varphi_\rho \circ \Pi$.
	
	Now, we define the class of potentials that we consider in our first main result.
	
	\begin{definition}\label{def.admissible}
		A continuous potential $\phi:\Lambda \longrightarrow \mathbb{R}$ is called admissible if satisfies:
		\begin{itemize}
			\item[$(C_{1})$] The  potential $\Psi = \varphi_\rho \circ \Pi$ associated to $\varphi:=\phi\circ h^{-1}:\Sigma \rightarrow \mathbb{R}$ is locally H\"{o}lder  continuous.
			\item[$(C_{2})$]  There exists a natural $n \in \mathbb{N}$ such that
			\[
			P_{top}(\phi)>\sup_{X\in\Lambda}\frac{\phi_n(X)}{n},
			\]
			where $\phi_{n}(X)=\sum_{k=0}^{n-1}\phi(F^{k}(X))$.
		\end{itemize}
	\end{definition}
	
	\begin{remark}\label{Remark2}
		If $\psi=\phi_n/n$ then $P_{top}(\psi)=P_{top}(\phi)$ and the equilibrium states for $\psi$  and $\phi$ coincide. Thus, we can assume without loss of generality that $n=1$. Another remark about admissible potentials is that since the condition $(C_2)$ is open in the $C^0$-topology and the topological entropy of $F$ is positive, it is easy to see that $\phi$ satisfy $(C_2)$ if $\sup |\phi|$ is small enough. 
	\end{remark}
	
	\begin{remark}
		Let $\mu$ be an $F$-invariant probability measure on $\Lambda$. Denote by  $\lambda^{c}_{\mu}=\int \log|DF|_{E^{c}}| d\mu$ the central Lyapunov exponent associated to $\mu$. 
		Using arguments similar to [\cite{LOR}, Lemma 4.3 and Corollary 4.4],  it is possible to show that the central Lyapunov exponents of the equilibrium states associated to admissible potentials are negative.
	\end{remark}

	Now, we state our first main result:
	
	\begin{theorema}\label{maintheorem}
		Let $F:\Lambda \longrightarrow \Lambda$ be the horseshoe map. If $\phi:\Lambda \longrightarrow \mathbb{R}$ is an admissible potential, then  there exists a unique equilibrium state for $\phi$.
	\end{theorema}
	
	In \cite{RS2}, the authors proved that if $\phi$ is a H\"{o}lder  continuous potential with small variation, i.e., $\sup \phi - \inf \phi$ is small enough,  then $\phi$ has a unique equilibrium state. The result of Theorem A differs from \cite{RS2} due to the strategy of the proof of uniqueness of equilibrium states and also due to the fact that the class of admissible potential contains potentials that do not satisfy the condition of small variation. Indeed, in Section~\ref{Example}, we construct new examples of potentials that are admissible and, consequently, have a unique equilibrium state. In fact, we show that if $\phi$ is a H\"{o}lder  potential that is  constant on the set $\{(x,y,z); z<c\}$, for some $5/6<c<1$  and $\sup \phi = \phi(Q)$, then there exists a non-empty open interval $I\subset \mathbb{R}$ such that if $t \in I$ then  $t\phi$ is admissible. A precise description of $I$ is given in Proposition~\ref{Example Prop.}.

	Now we reproduce the definition of a map $G$, as introduced in \cite{RS}, which is related to the projection of $F^{-1}$ in two center-stable planes. We consider an abstract space consisting of three rectangles $S_{1}$, $S_{2}$ and
	$S_{3}$ defined as follows:
	\[
	S_{1} = [0,\lambda_{0}] \times [0,1] \times \{0\},
	\]
	\[   
	S_{2} = [3/4 - \lambda_{0},3/4] \times [0,\sigma] \times \{0\},
	\]
	\[
	S_{3} =	[0,\lambda_{0}] \times [0,1] \times \{5/6\}.
	\]
	
	The rectangles are inside two planes 
	$P_{0}=[0,1]\times [0,1] \times \{0\}$ and $P_{1}=[0,1]\times [0,1] \times \{5/6\}$. Let $g_{0} : [0,1] \to \mathbb{R}$, $g_{0}(y) = f^{-1}(y)$ and let $g_{1} : [0,1] \to \mathbb{R}$, $g_{1}(y) = 1 - \sigma^{-1}y$. Let $\alpha = \lambda_{0}^{-1}$.
	
	We define a map $G : \cup_{i=1}^{3} S_{i}
	\to P_{1} \cup P_{2}$ by its restrictions $G_{i}$ to each rectangle $S_{i}$ as follows:	
	
	\[
	G_{1}(x,y,z) = (\alpha x, g_{0}(y),0),
	\]
	\[   
	G_{2}(x,y,z) = (\alpha(3/4 - x), g_{1}(y), 5/6),
	\]
	\[
	G_{3}(x,y,z) = (\alpha x, g_{0}(y),0).
	\]
	The behavior of the map $G$ is similar on $S_{1}$ and $S_{3}$. In these rectangles, the vertical direction is expanding. In the horizontal direction, $G$ sends the points from the right side to the left side, except for the extreme points whose the $x$ coordinates are fixed. In the rectangle $S_{2}$ both directions are expanding. 
	
	We consider $\Omega$ the maximal invariant set of the union of the rectangles:
	\[
	\displaystyle \Omega : = \bigcap_{n \in \mathbb{N}} G^{-n}\bigg{(} \bigcup_{i=1}^{3} S_{i} \bigg{)},
	\]
	and from now on we call $G$ the restriction $G_{|\Omega} : \Omega \to 
	\Omega$. See details of the map $G$ in \cite{RS} and \cite{RS2}, as the proof that this map is non-uniformly expanding. We also observe that $G$ is strongly mixing.

	
	The map $G$ is very important because we can use results for non-uniformly expanding maps and transfer them to the horseshoe $F$, as in the following main result, which allows potentials with big variation. This is a novelty with respect to previous results. 
	
	\begin{theorema} \label{transfer}
		Let $F : \Lambda \to \Lambda$ be the horseshoe map and $G:\Omega \to \Omega$ be the projection of $F^{-1}$ as defined above. The system $(F^{-1}, \phi)$ has uniqueness of equilibrium state if, and only if, the system $(G,\phi_{|\Omega})$ does so.
	\end{theorema}
	
	We define a class of potentials that we will use to obtain uniqueness of equilibrium states in our third main result. They are related to the non-uniformly expanding map $G$ as follows.

	\begin{definition}\label{def.hyperbolic}
		A continuous potential $\phi:\Lambda \longrightarrow \mathbb{R}$ is called projective hyperbolic if the following holds:
		\begin{itemize}
			\item[$(D_{1})$] The restriction $\varphi:=\phi_{|\Omega}$ is H\"{o}lder and there exists a measurable function $u:\Omega \to \mathbb{R}$ such that  $\varphi \geq u - u \circ G$.
			\item[$(D_{2})$] If $\mathcal{E}$ denotes the set of $G$-invariant expanding measures on $\Omega$, we have that
			\[
			\sup_{\nu \in \mathcal{E}^{c}}\bigg{\{}\int \varphi d\nu\bigg{\}} < h(G) - \sup_{\nu \in \mathcal{E}^{c}}\{h_{\nu}(G)\}.
			\]  
		\end{itemize}
	\end{definition}
	The class of projective hyperbolic potentials contains some potentials with big variation, as will be seen in Section \ref{Example Hyperbolic}.

	We can now state our third main result as a class of potentials that includes potentials that do not satisfy the condition of small variation and also constant on the fibers in the $z$-axis:
	
	\begin{theorema}\label{hyperbolic}
		Let $F:\Lambda \longrightarrow \Lambda$ be the horseshoe map. If $\phi:\Lambda \longrightarrow \mathbb{R}$ is a projective hyperbolic potential, then  there exists a unique equilibrium state $\mu_{\varphi
		}$ for the system $G$ with respect to the potential $\varphi = \phi_{|\Omega}$. It is equivalent to existing a unique equilibrium state for the horseshoe $F$.
	\end{theorema}
	
	We observe that in Section \ref{Rectangles} we show that if $\phi$ is a projective hyperbolic potential, then $\varphi = \phi_{|\Omega}$ is a hyperbolic potential (Proposition \ref{projectiveexpanding}) and since $G$ is a non-uniformly expanding map that is also strongly mixing (with density of pre-images of the points in the expanding set), we can obtain uniqueness of equilibrium states by \cite{RV}[Theorem 2].
	
	In the following, we state our fourth main result, which gives a condition for uniqueness of potentials with supremum at the fixed point $Q$ and it is less than the pressure.
    
    \begin{theorema}\label{ThQ}
     Let $F : \Lambda \to \Lambda$ be the horseshoe, $\eta$ the unique measure of maximal entropy for the non-uniformly expanding map $G : \Omega \to \Omega$ and $\phi$ a H\"older potential such that it holds that $\sup \phi_{|\Omega}  = \phi_{|\Omega}(Q)  < P(\phi_{|\Omega})$. Then, we obtain uniqueness of equilibrium state. Additionally, 	we have either every H\"older potential $\varphi$ such that $\sup \varphi_{|\Omega} < P(\varphi_{|\Omega})$ has uniqueness of equilibrium state or there exists a H\"older potential $\phi$ such that $\phi_{|\Omega}(Q)< \sup \phi_{|\Omega} < P(\phi_{|\Omega})$ and $\phi$ does not have uniqueness of equilibrium states.
    \end{theorema}
    
	The Theorems \ref{maintheorem} and \ref{hyperbolic} of this article differs from Theorem \ref{RS2} because of the potentials in Theorem \ref{maintheorem} and \ref{hyperbolic} include examples with big variation and also the strategy used to obtain uniqueness can be applied for other classes of systems. The same can be said for our Theorem \ref{hyperbolic} in this article, by comparing with the Theorem B in \cite{RS}.
	
	The work is organized as follows:  Section~\ref{Construction of I.E} is devoted to the  construction of  the symbolic structures that allow us to make use of Sarig's theory for countable shifts and show our first main  theorem. In the Section~\ref{Proof Main Theorem} we prove the first main result and we describe new classes of potentials that are admissible in Section~\ref{Example} and fits in the hypothesis of Theorem \ref{maintheorem}. Section \ref{equivalence} is devoted to the proof of Theorem \ref{transfer}. In Section \ref{versus} we prove an equivalence between hyperbolic and expanding potentials (see definitions there), which will be useful to show that the restrictions of projective hyperbolic potentials are hyperbolic potentials. In Section \ref{Rectangles} we prove our second main result. Section \ref{Q} is devoted to prove our third main result. Finally, in Section \ref{Example} we give a family of examples of admissible potentials and Section \ref{Example Hyperbolic} gives a construction of projective hyperbolic potentials.

	\section{Constructing the Inducing Scheme}\label{Construction of I.E}
	\vspace{0.50cm}
	
	In this section, we construct an inducing scheme associated to the horseshoe map $F$ that admits a symbolic representation as a shift over a countable alphabet.
	
	We define the sequences
	
	\[
	d_{n}^{+}(w) = \frac{\sharp\{ k \ ; \ w_{k}=1
		, \ 0\leq k \leq n-1\}}{n}.
	\]
	For each $\alpha>0$ we consider  the subset
	\begin{equation}\label{Sigmaalpha}
		\Sigma_{\alpha}^{+}=\{w\in [1] ;  \  \overline{\lim_{n}} \ d^{+}_{n}(w)>\alpha\}.
	\end{equation}
	
	The set $\Sigma_{\alpha}^{+}$ is composed by the sequences  with upper frequency of  digits 1  greater than $\alpha$. 
	
	
	We define on $\Sigma_{\alpha}^{+}$ the \emph{$\alpha$-return function}
	$\rho:\Sigma_{\alpha}^{+}\rightarrow\mathbb{N}$ by
	
	\[
	\rho(w)=\min\{k > 1 \ ; \ w_{k-1}=1  \  \  \textrm{and}  \  \  \ d_{k}^{+}(w)>\alpha\}.
	\]
	Note that $\Sigma_\alpha^{+}$ can be decomposed into level sets of the function $\rho$,
	$$
	\Sigma_{\alpha}^{+}= \bigcup_{i}\Sigma_{i},
	$$ where $\Sigma_{i}$ is given by $\Sigma_i=\{w\in\Sigma_{\alpha}^{+}; \rho(w)=i\}$.
	
	Now, we consider the set 
	\[
	\Sigma_{\alpha}^{-}=\{w\in [1] \,; \, \exists \, n_{k}\rightarrow +\infty, \,  \sigma^{-(n_{1}+n_{2}+\dots +n_{k})}(w)\in \Sigma_{n_{k}}, \forall k\geq 1 \}.
	\]
	
	We define the set $\Sigma_{\alpha} :=\Sigma_{\alpha}^{+}\cap \Sigma_{\alpha}^{-}$, which is  invariant under $\sigma^{\rho}$. By using Item 2 of Theorem \ref{t.LOR}, we have that for any $\omega \in \Sigma_{\alpha}$, the set $h^{-1}(\omega)$ is a single point, i.e.,  $\Sigma_{\alpha}\subset \Sigma=h(\hat{\Lambda})$.\\
	
	We define  the \emph{induced map} $T:\Sigma_{\alpha}\rightarrow \Sigma_{\alpha}$  associated to $\rho$ as
	$$
	T(\omega)=\sigma^{\rho(\omega)}(\omega).
	$$
	
	Consider the \textit{tower} associated to $\rho$, defined by
	\begin{equation}\label{tower}
		W=\bigcup_{i>1}\bigcup^{i-1}_{k=0}\sigma^{k}(\Sigma_{i}).
	\end{equation}
	Let $\nu$ be a $T$-invariant measure such that $\int \rho \, d\nu<\infty$, then we define the $\sigma$-invariant measure, that will be called
	\textit{lifted} of $\nu$ by
	\begin{equation}\label{MP}
		\mathcal{L}(\nu)(A):= \left( \int \rho d\nu \right) ^{-1}\sum_{i>1}\sum^{i-1}_{k=0}\nu(\sigma^{-k}(A)\cap \Sigma_{i})
	\end{equation}
	for $A\subset\Sigma_{11}$.
	
	The map $\mathcal{L}$ that acts on $\mathcal{M}_{T}(W)$ is not necessarily surjective over $\mathcal{M}_{\sigma}(\Sigma_{11})$. Given a measure  $\mu\in\mathcal{M}_{\sigma}(\Sigma_{11})$, if there exists $\nu$ in $\mathcal{M}_{T}(W)$ such that $\mathcal{L}(\nu)=\mu$, then we say that $\mu$
	is a \textit{liftable measure}.
	We denote this class of measures by
	\begin{equation}\label{Me. Lif}
		\mathcal{M}_{L}(\sigma,W):=\{\mu\in\mathcal{M}_{\sigma};  \  \exists \ \nu \ , \  \mathcal{
			L}(\nu)= \mu , \,  \,  \mu(W)=1\, \}.
	\end{equation}
	Given a measurable set $A\subset\Sigma$ and $i>1$, we define
	\[
	e(i,A):=\frac{\sharp\{0\leq k \leq i-1; \ \sigma^{k}(\Sigma_{i})\cap A\neq \emptyset \}}{i}.
	\]
	The following lemma will be useful to ensure that a measure is liftable.
	\begin{theorem}[\cite{PSZlifted}, Theorem 3.1]\label{Med. Lev.}
		An ergodic $\sigma$-invariant measure $\mu$ that satisfies $\mu(W)>0$ is a liftable measure if
		there exists a number $N\geq0$ and a subset $A\subset\Sigma$ such that
		\[
		\mu(A) > \sup_{i>N}e(i,A).
		\]
	\end{theorem}
	Let $\varphi:\Sigma \rightarrow \mathbb{R}$ be a potential, we define the \textit{induced potential} associated to  $\varphi$ by
	\begin{equation}\label{inducedpotential}
		\varphi_{\rho}(w):=\sum^{i-1}_{k=0}\varphi(\sigma^{k}(w))
	\end{equation}
	whenever $w\in \Sigma_{i}$.
	
	The generalized Ka\v{c}-Abramov formula gives an interesting relation between the entropy of the original system and the entropy of the induced system.
	\begin{proposition}[see \cite{Z}, Theorem 5.1 and \cite{PS},Theorem 2.3]
		If $\nu$ is a measure in $\mathcal{M}_T(W)$ such that $\int\rho \, d\nu<\infty$, then
		\[
		h_{\mathcal{L}(\nu)}(\sigma)\int \rho \, d\nu=h_{\nu}(T).
		\]
		Let $\varphi$ be a potential and $\varphi_{\rho}$ its induced potential. If $\int\varphi_{\rho} \,  d\nu<\infty$ then
		\[
		\int \varphi \, d\mathcal{L}(\nu) \int \rho \, d\nu =\int \varphi_{\rho} \, d\nu.
		\]
	\end{proposition}
	
	We define the quantity 
	\[
	P_{L}(\varphi):=\sup_{\mu\in \mathcal{M}_{L}(\sigma,W)}\{h_{\mu}(\sigma)+\int_{W}\varphi \,  d\mu\},
	\]
	that will be called \textit{relative pressure}. An invariant probability measure that attains the supremum is called  \textit{relative equilibrium state}.
	
	Now, it will be established a relation between the induced system and a countable Markov shift.
	

	Firstly, note that each $\Sigma_{k}$
	can be decomposed as a disjoint finite union of $k$-cylinders, that is, $\Sigma_{k} = \bigcup _{i=1}^{r_{k}} D_{i}^{k}$, where
	$D_{i}^{k}=[w_{0}w_{1} \dots w_{k-1}]$. 
	Let $\hat{\pi}$ be the map that sends a cylinder to a word 
	\[
	\hat{\pi}(D_{i}^{k})=w_{0}w_{1} \dots w_{k-1}.
	\]
	Using the new alphabet $S=\{D_{i}^{k} ;\,  k\geq 1,  1 \leq i \leq r_{k}\}$ we define the \emph{coding map}   $\Pi:S^{\mathbb{Z}}\rightarrow \Sigma_{\alpha}$ by the following amalgamation
	\[
	\Pi(\hat{w})=(\dots \hat{\pi}(D_{i_{k_{-1}}}^{k_{-1}}) \hat{\pi}(D_{i_{k_{0}}}^{k_{0}})\hat{\pi}(D_{i_{k_{1}}}^{k_{1}})\dots),
	\]
	where $\hat{w}=(D_{i_{k_{n}}}^{k_{n}})_{n\in\mathbb{Z}}\in S^{\mathbb{Z}}$.
	
	Observe that by construction $\Pi$ is well defined, and it is a conjugacy between $T$ and $\sigma$ (the countable Markov shift that acts on $S^{\mathbb{Z}}$). In fact, given a sequence $w\in \Sigma_{\alpha}$, by definition we have that there exists a (unique)  sequence $\hat{w}=(D_{i_{k_{n}}}^{k_{n}})_{n\in\mathbb{Z}}$ in $S^{\mathbb{Z}}$, such that $\sigma^{k_{n}}(w)\in D^{k_{n+1}}_{i_{k_{n+1}}}$ for $n\geq 0$  and $\sigma^{-k_{n}}(w)\in \Sigma_{k_{n}}$ for $n<0$ i.e.,  $\Pi(\hat{w})=w$.

	Now, we discuss the thermodynamic formalism of countable Markov shifts. The main references are \cite{PSZ}, \cite{Sar99}, \cite{Sar03} and \cite{Sar15}.
	
	Let $\sigma$ be the full shift on $S^\mathbb{Z}$ and $\Psi:S^{\mathbb{Z}} \rightarrow \mathbb{R}$ a potential, the \textit{$k$-variation}
	is defined by
	\[
	\mathrm{Var}_{k}(\Psi)=\sup_{[i_{-k+1}\dots i_{k-1}]} \ \sup_{\hat{w},\hat{s}\in [i_{-k+1}\dots i_{k-1}]}\{|\Psi(\hat{w})-\Psi(\hat{s})|\},
	\]
	where $[i_{-k+1} \dots i_{k-1}]$ is the two-sided cylinder  of all the sequences $w=(w_{i})_{i\in\mathbb{Z}}$ with $w_{-k+1}=i_{-k+i}$, \dots, $w_{0}=i_{0}, \dots, w_{k-1}=i_{k-i}$.
	
	We say that
	$\Psi$ has \textit{strongly summable variations} if
	\[
	\sum_{k\geq1}k\mathrm{Var}_{k}(\Psi)<\infty,
	\]
	and $\Psi$ is \textit{locally H\"{o}lder  continuous} if there exists $C>0$ and $a\in(0,1)$ such that for all $k\geq1$
	\[
	\mathrm{Var}_{k}(\Psi)\leq Ca^{k}.
	\]
	If $\Psi:S^{\mathbb{Z}}\rightarrow \mathbb{R}$ is locally H\"{o}lder  continuous  then the \textit{Gurevich pressure} of $\Psi$ is defined by
	\begin{equation}\label{Pre. Gur}
		P_{G}(\Psi,a):=\lim_{n\rightarrow \infty}\frac{1}{n}\log\sum_{\stackrel{\sigma^{n}(\hat{w})=\hat{w}}{\hat{w}\in [a]}}\exp(\Psi_{n}(\hat{w}))\mathcal{X}_{[a]}(\hat{w}),
	\end{equation}
	where $\mathcal{X}_{[a]}$ is the indicator function on the cylinder set $[a]$.
	
	In \cite{Sar99} it was proved that the limit \eqref{Pre. Gur} exists and it is independent from $a\in S$. We denote it by $P_{G}(\Psi)$.
	
	Let $\mathcal{M}_{\sigma}(S^{\mathbb{Z}})$ be the set of  $\sigma$-invariant Borel probabilities on $S^{\mathbb{Z}}$ and
	\[
	\mathcal{M}_{\sigma}(\Psi):=\{\eta\in\mathcal{M}_{\sigma}(S^{\mathbb{Z}});\int\Psi \, d\eta>-\infty\}.
	\]
	A measure $\eta_{\Psi}$ in $\mathcal{M}_{\sigma}(S^{\mathbb{Z}})$ is an equilibrium state for a potential $\Psi$ if satisfies
	\[
	h_{\eta_{\Psi}}(\sigma)+\int\Psi d\eta_{\Psi} = \sup_{\eta\in\mathcal{M}_{\sigma}(\Psi)}\{h_{\eta}(\sigma)+\int\Psi d\eta\}.
	\]
	We say that $\eta$ is a Gibbs measure for $\Psi$ if there exists a constant $C$ such that for a cylinder $[i_{0}\dots i_{n-1}]$ in $S^{\mathbb{Z}}$
	and  $\hat{w}\in[i_{0}\dots i_{k-1}]$ we have
	\[
	C^{-1}\leq\frac{\eta([i_{0}\dots i_{n-1}])}{e^{(\Psi_{n}(\hat{w})-nP_{G}(\Psi))}}\leq C.
	\]
	The next result can be found in \cite{PSZ}[Theorem 3.1].
	\begin{theorem}\label{Pri. Var}
		Assume that $\sup_{\hat{w}\in S^{\mathbb{Z}}}\Psi(\hat{w})<\infty$ and $\Psi$ has strongly summable variations. Then
		\begin{enumerate}
			\item the variational principle for $\Psi$ holds
			\[
			P_{G}(\Psi)=\sup_{\eta\in\mathcal{M}_{\sigma}(\Psi)}\{h_{\eta}(\sigma)+\int\Psi \, d\eta\}.
			\]
			\item If $P_{G}(\Psi)<\infty$, then there exists a unique $\sigma$-invariant ergodic Gibbs measure $\eta_{\Psi}$ for $\Psi$.
			\item If furthermore, $h_{\eta_{\Psi}}(\sigma)<\infty$, then $\eta_{\Psi}\in\mathcal{M}_{\sigma}(\Psi)$ and it is the unique
			equilibrium measure for $\Psi$.
		\end{enumerate}
	\end{theorem}
	We say that an invariant Borel probability measure $\eta$ associated to the continuous transformation  $T$ has
	\textit{exponential decay of correlations} for a class $\mathcal{H}$ of functions
	if there exists $0<\theta<1$ such that, for any $h_{1}, h_{2}\in \mathcal{H}$,
	\[
	\Big| \int h_{1}(T^{n}(x))h_{2}(x) \, d\eta - \int h_{1}(x) \, d\eta \int h_{2}(x) \, d\eta \Big| \leq K\theta^{n},
	\]
	for some $K=K(h_{1},h_{2})>0$. \\
	The measure $\eta$ satisfies the \textit{central limit theorem} (CLT) for functions in $\mathcal{H}$ if there exist $\sigma \in \R$ such that $\frac{1}{\sqrt{n}} \sum^{n-1}_{i=0}(h(T^{i}(x))- \int h \,  d\eta)$ converges in law to a normal distribution $\mathcal{N}(0,\sigma)$ for any $h\in \mathcal{H}$. 
		\begin{proposition}[\cite{PSZ} Theorem 4.7]\label{CLT}
			Assume that $P_{G}(\Psi)$, $\sup_{\hat{w}\in S^{\mathbb{Z}}}(\Psi(\hat{w}))<\infty$ and that $\Psi$ is locally H\"{o}lder  continuous.
			If $h_{\eta_{\Psi}}(\sigma)<\infty$, then the measure $\eta_{\Psi}$ has exponential decay of correlations and satisfies the
			CLT with respect to the class of bounded locally H\"{o}lder continuous.
	\end{proposition}
	
	
	\section{Proof of Theorem \ref{maintheorem}}\label{Proof Main Theorem}
	\vspace{0.50cm}
	
	The proof of Theorem \ref{maintheorem} 
	is divided into two steps. The first one  is to show the uniqueness of equilibrium state for the induced map and the second one is to obtain the uniqueness for the horseshoe map.\\
	
	Consider the constant $\alpha \in(0,\frac{2}{3})$ and the inducing scheme $(\rho,\Sigma_{\alpha})$, obtained as in the Section \ref{Construction of I.E}, the next
	result gives us sufficient condition for a potential $\varphi$ to admit a unique relative equilibrium state.
	\begin{theorem}\label{Uni.Esq.}
		Let $\varphi:\Sigma \rightarrow \mathbb{R}$ be a continuous potential and $\varphi_\rho$ be the induced potential satisfying the following conditions: 
		\begin{itemize}
			\item[$(P_{1})$] The potential  $ \Psi:=\varphi_\rho\circ \Pi$ is locally H\"{o}lder  continuous.
			\item[$(P_{2})$] There exist a natural $n\in \mathbb{N}$ such that $P_{L}(\varphi)>\sup\frac{\varphi_{n}}{n}$.
		\end{itemize}
		Then, there exists a unique relative equilibrium state $\nu_{\varphi}$. Moreover, the measure $\nu_{\varphi}$ has exponential decay of correlations and satisfies the
		(CLT) with respect to class of functions whose induced functions on $\Sigma_{\alpha}$ are
		bounded H\"{o}lder  continuous.
	\end{theorem}
	
	Note that the pressure $P_{L}(\varphi)$ is finite, so we take $\varphi_\rho=\overline{\varphi-P_{L}(\varphi)}$.
	\begin{claim}\label{afirmation1}
		If $\varphi$ satisfies $(P_{2})$, then there exists a constant $\epsilon>0$ such that
		\begin{equation}\label{I}
			\sum_{i>1}i\sup_{w\in \Sigma_{i}}e^{\varphi_\rho(w)+i\epsilon}< \infty.
		\end{equation}
	\end{claim}
	\begin{proof}[Proof of Claim]
		By the property $(P_{2})$ and Remark \ref{Remark2} we get $\epsilon_{0}>0$ such that $P_{L}(\varphi)>\sup\varphi +\epsilon_{0}$ and thus for $P:=P_{L}(\varphi)$, $w\in\Sigma_{i}$ and all $i$ it occurs that
		\[
		-\epsilon_{0}> \frac{\sum^{i-1}_{k=0}\varphi(\sigma^{k}(w))}{i}-P.
		\]
		This implies that $\sum^{i-1}_{k=0}\varphi(\sigma^{k}(w))-iP<-i\epsilon_{0}$
		and $
		e^{\varphi_\rho(w)}<e^{-i\epsilon_{0}}$.
		Since the estimate is independent of $w$, it is possible to obtain $\epsilon$ small enough such that
		\begin{equation}\label{4}
			\sup\{e^{\varphi_\rho(w)+ i \epsilon} ; w\in \Sigma_{i}\}\leq e^{-i\epsilon_{0}}.
		\end{equation}
		Hence, by summing the expression \eqref{4} over $i$ we get
		\[
		\sum_{i}i \sup\{e^{\varphi_\rho(w) +  i\epsilon} ; w\in \Sigma_{i}\}\leq \sum_{i}ie^{-i\epsilon_{0}} < +\infty.
		\]
		Thus proving the Claim \ref{afirmation1}.
	\end{proof}
	We now consider the potential $\Psi=\varphi_\rho\circ\Pi$. 
	\begin{claim}\label{afirmation2}
		If $\varphi$ satisfies $(P_{1})$ and $(P_{2})$, then $\Psi$ the  Gurevich pressure is finite. 
	\end{claim}
	\begin{proof}[Proof of Claim]
		By the condition $(P_{1})$ the Gurevich Pressure $P_{G}(\Psi)$ is well defined. Let $n$ be positive integer and a cylinder  $[D^{l_{1}}_{i_{1}}...D^{l_{n}}_{i_{n}}]$ in $S^{\mathbb{Z}}$, there exists a unique  sequence $w\in\Pi([D^{l_{1}}_{i_{1}}...D^{l_{n}}_{i_{n}}])$ such that $T^{n}(w)=w$.  Thus, by fixing $a>1$ and $1\leq j\leq r_{a}$ we get 
		\begin{eqnarray}\label{II}
			\left(\sum_{i>1}\sup_{w\in\Sigma_{i}}e^{\varphi_{\rho}(w)} \right)^{n} &=& \sum_{i_{1},...i_{n}}C_{i_{1}...i_{n}}
			\sup_{w\in \Sigma_{i_{1}}}e^{\varphi_{\rho}(w)}...\sup_{w\in \Sigma_{i_{n}}}e^{\varphi_{\rho}(w)} \nonumber \\
			&\geq&\sum_{i_{2},...i_{n}}C_{a,i_{2}...i_{n}}
			\sup_{w\in \Sigma_{a}}e^{\varphi_{\rho}(w)}\sup_{w\in \Sigma_{i_{2}}}e^{\varphi_{\rho}(w)}...\sup_{w\in \Sigma_{i_{n}}}e^{\varphi_{\rho}(w)}  \nonumber \\
			&\geq& \sum_{\stackrel{T^{n}(w)=w}{w\in D ^{a}_{j}}} e^{\sum^{n-1}_{k=0}\varphi_{\rho}(T^kw)}. \nonumber
		\end{eqnarray}
		
		Then by (\ref{I}) and \eqref{II}, we have
		\begin{eqnarray}
			P_{G}(\Psi) &=& \lim_{n\rightarrow \infty}\frac{1}{n}\log
			\sum_{\stackrel{\sigma^{n}(\hat{w})=\hat{w}}{\hat{w}\in [D^{a}_{j}]}}e^{\sum^{n-1}_{j=0}\Psi(\sigma^{j}(\hat{w}))} \nonumber\\
			&=&\lim_{n\rightarrow \infty}\frac{1}{n}\log
			\sum_{\stackrel{T^{n}(w)=w}{w\in D^{a}_{j}}}e^{\sum^{n-1}_{k=0}\varphi_{\rho}(T^{k}(w))} \nonumber\\
			&\leq & \lim_{n\rightarrow \infty}\frac{1}{n}\log
			\left( \sum_{i>1}\sup_{w\in \Sigma_{i}}e^{\varphi_{\rho}(w)} \right)^{n} < +\infty.\nonumber
		\end{eqnarray}
	\end{proof}
	
	Note that by using the same arguments we prove that the Gurevich pressure of the potential  $\Psi_{\delta}$ on $S^{\mathbb{Z}}$, given by $\Psi_{\delta}(\hat{w})= \varphi_{\rho}(w) + \delta \rho(w)$ is finite for some $\delta>0$ small enough. A potential $\Psi$ that satisfies this is called \textit{positive recurrent potential}.
	
	\begin{proof}[Proof of Theorem \ref{Uni.Esq.}]
		By Theorem \ref{Pri. Var}, we have the variational principle
		\[
		P_{G}(\Psi)=\sup_{\eta\in\mathcal{M}_{\sigma}(\Psi)}\{h_{\eta}(\sigma)+\int\Psi \, d\eta\}.
		\]
		Moreover, there exists a unique Gibbs measure $\eta_{\Psi}$ for $\Psi$, so the measure $\nu=\Pi_{\ast}\eta_{\Psi}$ also has this property with respect to the potential $\varphi_{\rho}$, that is, there exists a constants
		$K>0$ such that for $n>1$ and $\Sigma_{n}= \cup_{i=1}^{r_{n}}D_{i}^{n}$ we have
		\begin{equation}\label{Pro. Gib}
			K^{-1}\leq \frac{\nu(D_{i}^{n})}{e^{\varphi_{\rho}(w)-P}}\leq K,
		\end{equation}
		for $w\in D_{i}^{n}$ and $P=P_{G}(\Psi)$. \\
		Now, we define 
		\[
		c(\alpha) = \limsup \frac{1}{n} \log \left[ \sum_{k\leq \alpha n} {n \choose k }  \right] .
		\]
		By using standing theorem for $n!$, we can easily cheek that $\lim_{\alpha\rightarrow 0}c(\alpha)= 0$. Choose $\alpha>0$, such that $c(\alpha)< P_{L}(\varphi) - \sup(\varphi)$. 
		We observe that
		\begin{eqnarray}
			\limsup \frac{1}{n} \log r_{n}  & \leq & \limsup \frac{1}{n} \log \sharp\{ w=(w_{0}w_{1}...) \, ; \ d_{n-1}(w)\leq \alpha \}  \\ \nonumber
			& < & c(\alpha)< P_{L}(\varphi) - \sup(\varphi).  \\ \nonumber
		\end{eqnarray}
		Thus, for $n$ large enough we have that $r_{n} < e^{\epsilon n }$.
		Hence by summing \eqref{Pro. Gib} with respect to $n>1$ and using \eqref{I}, we have
		\begin{eqnarray}\label{II}
			\int_{\Sigma_{\alpha}}\rho \, d\nu &=& \sum_{n>1} n \nu (\Sigma_{n}) \nonumber\\
			&=& \sum_{n>1} n \sum_{i=1}^{r_{n}} \nu( D_{i}^{n}) \nonumber\\
			&\leq& \frac{K}{e^{P}} \sum_{n>1} n r_{n} \sup_{w\in D_{i}^{n}}\{e^{\varphi_{\rho}(w)+ \epsilon n}\}  < +\infty .
		\end{eqnarray}
		By using \eqref{II} and Abramov-Kac formulas, we have
		\begin{eqnarray}
			h_{\eta_{\Psi}}(\sigma)  & = & h_{\nu}(T) \nonumber\\
			& = & \int \rho \, d\nu \cdot h_{\mathcal{L}(\nu)}(\sigma) < + \infty. \nonumber
		\end{eqnarray}
		Since $ P_{L}(\varphi) > -\infty$ and $\int \varphi \, d\mathcal{L}(\nu) > -\infty$, we have
		\begin{eqnarray}
			\int\Psi \, d\eta_{\Psi} & = & \int \varphi_{\rho} \, d\nu\nonumber \\
			& = & \left(\int \rho \, d\nu \right) \cdot \int \varphi - P_{L}(\varphi) \, d\mathcal{L}(\nu) > - \infty . \nonumber
		\end{eqnarray}
		Hence $\eta_{\Psi}\in\mathcal{M}_{\sigma}(\Psi)$, and by using Theorem \ref{Pri. Var} once again, we obtain that the measure $\eta_{\Psi}$ is the unique equilibrium state of $\Psi$. However, we can define the measure $\mathcal{L}(\nu)$ that is contained in $\mathcal{M}_{L}(\sigma,W)$. By using that $\Psi$ is a positive recurrent potential and  [\cite{PS}, Theorem 4.4] we obtain that the
		measure $\mathcal{L}(\nu)$ is the unique relative equilibrium state for $P_{L}(\varphi)$.
		
		In order to complete the proof of Theorem~\ref{Uni.Esq.}, we just need to check the properties of Proposition~\ref{CLT}. We make use of the following lemma.
		\begin{lemma}\label{III}
			The measure $\nu$ has exponential tail, that is, there exist $C>0$ and $0 < \vartheta < 1$ such that for all $n>0$, 
			\[
			\nu(\{w\in \Sigma_{\alpha}; \rho(w)\geq m\}) \leq C \vartheta^{n}.
			\]
		\end{lemma}
		\begin{proof}
			Note that $\nu(\{w\in \Sigma_{\alpha}; \rho(w)\geq m\})  =  \sum_{n\geq m}^{\infty}\sum_{i=1}^{r_{n}} \nu( D_{i}^{n})$, then we obtain the result in a similar way to \eqref{II}.
		\end{proof}
		
		It follows that by Lemma \ref{III} and the results [\cite{Y}, Theorem 2, Theorem 3] that the measure $\mathcal{L}(\nu)$ has exponential decay of correlations and satisfies the
		(CLT) with respect to class of functions whose induced functions on $\Sigma_{\alpha}$ are
		bounded H\"{o}lder  continuous.
	\end{proof}

	\begin{proof}[ Proof of Theorem~\ref{maintheorem}]
		
		By using Theorem \ref{t.LOR}, we have that every continuous potential $\phi$ has an equilibrium state. Assume that $\mu_1$ and $\mu_2$ are ergodic equilibrium states for $F$ and let us prove that  $\mu_1=\mu_2$.
		
		By hypothesis $(C_{2})$, there exists a natural $n$ such that $P(\phi)>\sup \frac{\phi_n}{n}$, so
		\[
		\phi(Q)\leq\sup \frac{\phi_n}{n} < P(\phi).
		\]
		Then, the Dirac measure $\delta_Q$, can not be an equilibrium state for $\phi$. Analogously, we can obtain the same for $\delta_P$. Therefore, by Lemma \ref{Red. Med.}, we can consider a constant  $\alpha>0$ small enough such that $\mu_i(R_1)>\alpha$ for $i=1,2$.
		
		Now, we consider the  inducing scheme $(\rho,\Sigma_{\alpha})$ such that  $\Sigma_\alpha \subset \Sigma_{11}$ is the subset composed by the sequences with frequency of symbols 1's  at least $\alpha$, as defined in (\ref{Sigmaalpha}), and the associated
		tower $W$ as in (\ref{tower}).
		
		For each $i=1,2$, we denote by $\nu_i$ the push-forward of $\mu_i$ with respect to semi-conjugacy $h$, i.e., the measure defined on $A\subset \Sigma_{11}$ by
		$$
		\nu_i(A)=\mu_i(h^{-1}(A)).
		$$
		The measure $\nu_i$ is ergodic and satisfies  $\nu_i([1])=\mu_i(R_1)>\alpha>0$, so since $W$ is an $\sigma$-invariant set, by ergodicity of $\nu_i$ we have $\nu_{i}(W)=1$. Hence, applying the Lemma \ref{Med. Lev.} for $A=[1]$, we conclude that $\nu_{i}$ is a liftable measure, moreover it is
		a relative equilibrium state for the potential $\varphi=h_{\ast}\phi$. In the other hand, the potential $\phi$ satisfies the conditions $(C_{1})$ and $(C_{2})$,
		so $\varphi$ satisfies the properties $(P_{2})$ and $(P_{1})$. Then, by Theorem \ref{Uni.Esq.}, there exists a unique relative equilibrium state, so $\nu_{1}=\nu_{2}$. Therefore, we conclude that $\mu_{1}=\mu_{2}$. The uniqueness is proven.
	\end{proof}
	
	\section{Proof of Theorem \ref{transfer}}\label{equivalence}
	
	In this section, we prove that uniqueness of equilibrium state for the restriction of the potential to $\Omega$ and the map $G : \Omega \to \Omega$ is equivalent to uniqueness of equilibrium state for the horseshoe map $F : \Lambda \to \Lambda$. The technique is based on the strategy given in the works \cite{RS} and \cite{RS2}. First, we prove that potentials $\phi$ with uniqueness for the restriction $\phi_{|\Omega}$ and such that they do not depend on the $z$ coordinate have uniqueness. After that, we prove that potentials without this condition have cohomologous satisfying this and the uniqueness is obtained in the general case.
	
	In order to make the paper self contained, we will prove our result following the proofs in the works \cite{RS} and \cite{RS2} along the same lines. All the strategy works here because the unique property used with respect to the restriction to $\Omega$ is the uniqueness obtained from the small variation of the potential and the non-uniform expansion of the map $G$. 
	
	As it is said in \cite{RS}, some of the following ideas are based in the constructions in \cite{CN}. We define the projection of the parallelepipeds $R_{0}$ and $R_{1}$ onto the planes $P_{0}$ and $P_{1}$, $\pi : R_{0} \cup R_{1} \to P_{0} \cup P_{1}$ by
	\[
	\pi(x,y,z) = \left \{ 
	\begin{array}{cc}
		(x,y,0), & \,\, if \,\,(x,y,z), \in R_{0}, \\
		(x,y,\frac{5}{6}) & \,\, if \,\, (x,y,z) \in R_{1}. 
	\end{array}
	\right.
	\]
	It is straightforward to check that $\pi$ is continuous, surjective and
	\[
	\pi \circ F^{-1} = G \circ \pi.
	\]
	
	\begin{lemma}[\cite{RS}, Lemma 5.1]\label{zero entropy}
		For each $X \in \Omega$ we have $h(F^{-1}, \pi^{-1}(X)) = 0$.
	\end{lemma}
	
	Since the projection $\pi$ is actually a semiconjugacy between the inverse horseshoe $F^{-1}$ and the map $G$, we have (see \cite{B})
	\[
	h_{\text{top}}(F^{-1}) \leq h_{\text{top}}(G) + \sup\{h(F^{-1}, \pi^{-1}(X)) \mid X \in \Omega\}. 
	\]
	By Lemma \ref{zero entropy} we get
	\[
	h_{\text{top}}(F^{-1}) \leq h_{\text{top}}(G). 
	\]
	On the other hand, because $\pi$ is a semiconjugacy, we have the other inequality immediately (see \cite{Bowen}), which gives us
	\[
	h_{\text{top}}(F^{-1}) = h_{\text{top}}(G). 
	\]
	Thus, we have proved that the topological entropies of these two maps are equal.
	
	\begin{proposition}[\cite{RS}, Proposition 5.2]\label{zero entropy}
		For $\varphi = \phi_{|\Omega} \circ \pi$ we have $P_{\text{top}}(F^{-1},\varphi) = P_{\text{top}}(G,\phi_{|\Omega})$.
	\end{proposition}
	\begin{proof}
		Given $\mu$ a $F^{-1}$-invariant probability measure, taking $\nu = \mu \circ \pi^{-1}$ a result by Ledrappier and Walters (see \cite{LW}) guarantees that
		\[
		h_{\mu}(F^{-1}) + \int \phi_{|\Omega} \circ \pi d\mu \leq  h_{\nu}(G) + \int \phi_{|\Omega} d\nu + \int_{\Omega} h(F^{-1}, \pi^{-1}(X)) d\mu(X) = h_{\nu}(G) + \int \phi_{|\Omega} d\nu \leq P_{\text{top}}(G),
		\]
		because by Lemma \ref{zero entropy} we have $h(F^{-1}, \pi^{-1}(X)) = 0$ for each $X \in \Omega$ and by semiconjugacy $\nu$ is $G$-invariant.
		
		Going back to the inequality, we have for $\varphi = \phi_{|\Omega} \circ \pi$ and any $F^{-1}$-invariant probability $\mu$
		\[
		h_{\mu}(F^{-1}) + \int \varphi d\mu \leq P_{\text{top}}(G) \implies P_{\text{top}}(F^{-1},\varphi) \leq P_{\text{top}}(G,\phi_{|\Omega}).
		\]
		In the other hand, $\pi$ is a semiconjugacy, which implies 
		(see \cite{Bowen})
		\[
		P_{\text{top}}(G,\phi_{|\Omega}) \leq P_{\text{top}}(F^{-1},\varphi) \implies P_{\text{top}}(G,\phi_{|\Omega}) = P_{\text{top}}(F^{-1},\varphi).
		\]
	\end{proof}
	
	\begin{theorem}\label{lift}
		Given a $G$-invariant Borel probability $\mu$, there exists an $F^{-1}$-invariant Borel probability $\overline{\mu}$ such that $h_{\overline{\mu}}(F^{-1}) \geq h_{\mu}(G)$.
	\end{theorem}
	
	We divide the proof of Theorem 
	\ref{lift} into five lemmas.
	
	The first lemma constructs the measure and a $\sigma$-algebra.
	\begin{lemma}\label{measure}
		Denote by $\mathcal{A}_{\Omega}$ the Borel $\sigma$-algebra in $\Omega$. Given a $G$-invariant probability $\mu$, there exists a $\sigma$-algebra $\mathcal{A}$ and a measure $\overline{\mu} : \mathcal{A} \to [0,1]$ depending on $\mu$.
	\end{lemma}
	\begin{proof}
		Let $\mathcal{A}_{\Omega}$ be the Borel $\sigma$-algebra in $\Omega$ and define $\mathcal{A}_{0}:=\pi^{-1}(\mathcal{A}_{\Omega})$. Thus, $\mathcal{A}_{0}$ is a $\sigma$-algebra on the horseshoe, whose elements are given by $A = \pi^{-1}(B)$, where $B$ is a Borelian in $\Omega$. Note that $\mathcal{A}_{0} \subset F^{-1}(\mathcal{A}_{0})$. If for each $n \in \mathbb{N}$ we define $\mathcal{A}_{n}: = F^{-n}(\mathcal{A}_{0})$, we have an increasing sequence of $\sigma$-algebras
		\[
		\mathcal{A}_{0} \subset \mathcal{A}_{1} \subset \mathcal{A}_{2} \subset \dots \subset \mathcal{A}_{n} \subset \dots
		\]
		For each $n$ define a probability measure $\overline{\mu}_{n} : \mathcal{A}_{n} \to [0,1]$ by
		\[
		\overline{\mu}_{n}(F^{-n}(A_{0}) = \mu(\pi(A_{0})) \,\, \text{for all} \,\, A_{0} \in \mathcal{A}_{0}.
		\]
		Note that $\overline{\mu}_{n}$ depends also on $\phi_{|\Omega}$. Using the semiconjugacy it is straightforward to prove that each $\overline{\mu}_{n}$ is invariant under $F^{-1}$. Finally, we consider $\mathcal{A}: = \bigcup_{n=0}^{\infty} \mathcal{A}_{n}$, which is a $\sigma$-algebra on $\Lambda$. Define $\overline{\mu} : \mathcal{A} \to [0,1]$ by
		\[
		\overline{\mu}(A) = \overline{\mu}_{n}(A) \,\,\, \text{if} A \in \mathcal{A}_{0}.
		\]
	\end{proof}
	
	It is easy to see that $\overline{\mu}$ is well-defined. The second lemma gives a criterium to show that the measure is $\sigma$-additive, which is classical in measure theory.
	
	\begin{lemma}\label{criterium}
		Let $\mathcal{A}$ be a $\sigma$-algebra and $\overline{\mu} : \mathcal{A} \to [0,1]$ such that
		\[
		\overline{\mu}(A_{1} \cup \dots \cup A_{n}) = \sum_{i=1}^{n}\overline{\mu}(A_{i}),
		\]
		for every finite family $\{A_i\}_{i=1}^{n}$ of disjoint sets in $\mathcal{A}$. If there exists a family $\mathcal{C} \subset \mathcal{A}$ such that
		\begin{itemize}
			\item $\mathcal{C}$ is compact: if $C_{1} \supset C_{2} \supset \dots \supset C_{n} \supset \dots$ is a sequence of sets in $\mathcal{C}$ then $\cap_{i=1}^{\infty} C_{i} \neq \emptyset$.
			
			\item $\mathcal{C}$ has the approximation property: for all $A \in \mathcal{A}$ we have
			\[
			\overline{\mu}(A) = \sup\{\overline{\mu}(C) \mid C \subset A, C \in \mathcal{C}\};
			\]
			then $\overline{\mu}$ is a probability measure on $\mathcal{A}$.
		\end{itemize}
	\end{lemma}
	
	The third lemma constructs a family satisfying the criterium.
	
	\begin{lemma}
		For the $\sigma$-algebra $\mathcal{A}$ in Lemma \ref{measure} there exists a family $\mathcal{C}$ satisfying the critirium in Lemma \ref{criterium} which warrants that the measure $\overline{\mu}$ in Lemma \ref{measure} is a probability on $\mathcal{A}$.
	\end{lemma}
	\begin{proof}
		We consider $K_{\Omega}$ the collection of all compact sets in $\Omega$ and we define $K_{0} : = \pi^{-1}(K_{\Omega})$, which is a subfamily of the $\sigma$-algebra $\mathcal{A}_{0}$.
		We define $\mathcal{C}: \bigcup_{n = 0}^{\infty} F^{-n}(K_{0})$ which is contained in $\mathcal{A}$ and it is a compact class with the approximation property.
		\begin{itemize}
			\item Let $C_{1} \supset C_{2} \supset \dots \supset C_{n} \supset \dots$ be sets of $\mathcal{C}$. If for every $i \in \mathbb{N}$ there exists $D_{i} \subset K_{\Omega}$ and there exists $n_{i} \in \mathbb{N}$ such that $C_{i} = F^{-n_{i}}(\pi^{-1}(D_{i}))$. Since $\pi^{-1}(D_{i})$ is a closed set in a compact space $\Omega$, $\pi^{-1}(D_{i})$ is compact. By the  continuity of $F^{-1}$, $(C_{i})_{i \in \mathbb{N}}$ is a family of nested compact sets, therefore $\bigcap_{n = 1}^{\infty} C_{n} \neq \emptyset$.
			
			\item Let $A \in \mathcal{A}$. Using the regulatiry of the measure $\mu$ on $\Omega$ we have $\mu(B) = \sup\{\mu(D) \mid D \subset B, D \in K_{\Omega}\}$ for all $D \in \mathcal{A}_{\Omega}$. By the definition of $\mathcal{C}$, for all $C \in \mathcal{C}$ there exists $D \in K_{\Omega}$ and $n \in \mathbb{N}$ such that $C = F^{-n}(\pi^{-1}(D))$. Thus, $\overline{\mu}(C) = \overline{\mu}(F^{-n}(\pi^{-1}(D))) = \mu(\pi(\pi^{-1}(D))) = \mu(D)$. Then
			\[
			\sup\{\mu(D) \mid D \subset B, D \in K_{\Omega}\} = \sup\{\overline{\mu}(C) \mid C \subset A, C \in \mathcal{A} \}.
			\]
			To finish it is enough to note that for all $A \in \mathcal{A}$, there exists $B \in K_{\Omega}$ such that $\overline{\mu}(A) = \mu(B)$. 
		\end{itemize}
		
		Then, by the criterion, we proved that $\overline{\mu}$ is a probability measure on $\mathcal{A}$. Since every $\overline{\mu}_{n}$ is invariant under $F^{-1}$, the measure $\overline{\mu}$ is also $F^{-1}$-invariant.
	\end{proof}
	
	The fourth lemma shows that the measure is $F^{-1}$-invariant on the Borel $\sigma$-algebra.
	
	\begin{lemma}
		The extension of the $\sigma$-algebra $\mathcal{A}$ coincides with the Borel $\sigma$-algebra $\mathcal{A}_{\Lambda}$ on $\Lambda$.
	\end{lemma}
	\begin{proof}
		Since $\pi$ is continuous we have that $\mathcal{A}$ is contained in the Borel $\sigma$-algebra. It is enough to prove that for all $X \in \Lambda$ there exists a fundamental system of neighborhoods of $X$ contained in $\mathcal{A}$.
		Let $X \in \Lambda$ and remember that $\rho < 1$ is the contraction factor in the direction of the $Z$-axis. For each $n \in \mathbb{N}$, by the uniform continuity of $G^{n}$ there exists $\delta_{n} > 0$ such that $d(z,w) < \delta_{n}$ implies $d(G^{n}(z), G^{n}(w)) < 1/n$.
		We define $B_{n}(X) = F^{-n}(\pi^{-1}(B(\pi(F^{n}(X)), \delta_{n})))$. Therefore, for every $n \in \mathbb{N}$ and $X \in B_{n}(X)$, we have $\text{diam}B_{n}(X) = \rho^{n} + 1/n$ which goes to zero when $n$ goes to infinity.
	\end{proof}
	
	The fifth lemma shows the inequality between the entropies.
	
	\begin{lemma}
		We have $h_{\overline{\mu}}(F^{-1}) \geq h_{\mu}(G)$.
	\end{lemma}
	\begin{proof}
		Using the well-known Brin-Katok theorem, we can compute the entropy by the formula:
		\[
		h_{\mu}(G) = \lim_{\epsilon \to 0} \limsup_{n \to \infty} \frac{1}{n} \log \frac{1}{\mu(B_{G}(X, \epsilon, n))},
		\]
		for $\mu$-almost every $X \in \Omega$, where $B_{G}(X, \epsilon, n)$ denotes the dynamical ball of $G$ at $X$, radius $\epsilon$ and length $n$. We note that, given $\epsilon > 0$, by the continuity of $\pi$ (on a compact space), there exists $0 \leq \delta \leq \epsilon$ such that
		\[
		\pi(B(Z,\delta)) \subset B(\pi(Z),\epsilon) \,\, \text{for every} \,\, Z \in \Lambda. 
		\]
		Using this fact it is straightforward to check that
		\[
		B_{F^{-1}}(Y, \delta, n) \subset \pi^{-1}(B_{G}(X, \epsilon, n)) \,\, \text{for every} \,\, Y \in \pi^{-1}(X).
		\]
		By the definition of the measure $\overline{\mu}$ we have
		\[
		\overline{\mu}(B_{F^{-1}}(Y, \delta, n)) \leq \overline{\mu}(\pi^{-1}(B_{G}(X, \epsilon, n))) = \mu(B_{G}(X, \epsilon, n))).
		\]
		If we consider $A \subset \Omega$ a full measure set with respect to $\mu$, then $\pi^{-1}(A) \subset \Lambda$ satisfies $\overline{\mu}(\pi^{-1}(A)) = \mu(A) = 1$. Thus,
		\[
		h_{\overline{\mu}}(F^{-1}) = \lim_{\delta \to 0} \limsup_{n \to \infty} \frac{1}{n} \log \frac{1}{\overline{\mu}(B_{F^{-1}}(Y, \delta, n))} \geq \lim_{\epsilon \to 0} \limsup_{n \to \infty} \frac{1}{n} \log \frac{1}{\mu(B_{G}(X, \epsilon, n))} = h_{\mu}(G).
		\]
	\end{proof}

	\begin{proposition}\label{lift}
		If the measure $\mu$ is an equilibrium state for the system $(G,\phi_{|\Omega})$ then $\overline{\mu}$ given by Theorem \ref{lift} is an equilibrium state for the system $(F^{-1},\varphi)$.
	\end{proposition}
	\begin{proof}
		Given an equilibrium state $\mu$ for the system $(G,\phi_{|\Omega})$, Theorem \ref{lift} gives an $F^{-1}$-invariant Borel probability $\overline{\mu}$ such that $h_{\overline{\mu}}(F^{-1}) \geq h_{\mu}(G)$. Hence,
		\[
		P_{\text{top}}(G) = h_{\mu}(G) + \int \phi_{\Omega} d\mu \leq h_{\overline{\mu}}(F^{-1}) + \int \varphi d\overline{\mu} \leq P_{\text{top}}(F^{-1}).
		\]
		Once $P_{\text{top}}(G) = P_{\text{top}}(F^{-1})$ and by a change of variables we get
		\[
		\int \phi_{\Omega} d\mu = \int \phi_{\Omega} \circ \pi d\overline{\mu} = \int \varphi d\overline{\mu},
		\]
		then we have that $\overline{\mu}$ is an equilibrium state for the system $(F^{-1},\varphi)$.
	\end{proof}
	\begin{proposition}\label{proj}
		If the measure $\mu$ is an equilibrium state for the system $(F^{-1},\varphi)$ then the measure given  by $\nu = \mu \circ \pi^{-1}$ is an equilibrium state for the system $(G,\phi_{|\Omega})$.
	\end{proposition}
	\begin{proof}
		Let $\mu$ an equilibrium state for the system $(F^{-1},\varphi)$ and define $\nu(A) = \mu(\pi^{-1}(A))$ for all Borelian $A$ in $\Omega$. Again by semiconjugacy, the probability $\nu$ is $G$-invariant. We have that
		\[
		P_{\text{top}}(F^{-1},\varphi) = h_{\mu}(F^{-1}) + \int \varphi d\mu \leq  h_{\nu}(G) + \int \phi_{|\Omega} d\nu \leq P_{\text{top}}(G),
		\]
		By Proposition \ref{zero entropy} we have $P_{\text{top}}(G,\phi_{|\Omega}) = P_{\text{top}}(F^{-1},\varphi)$, which implies that $\nu$ is an equilibrium state for the system $(G,\phi_{|\Omega})$.
	\end{proof}
	
	\begin{proposition}
		The system $(F^{-1},\varphi)$ has uniqueness of equilibrium state if, and only if, the system $(G,\phi_{|\Omega})$ does so.
	\end{proposition}
	\begin{proof}
		Assume that the system $(G,\phi_{|\Omega})$ has uniqueness of equilibrium state $\mu$. Proposition \ref{lift} guarantees that $\overline{\mu}$ is an equilibrium state for the system $(F^{-1}, \varphi)$. Let us suppose, by contradition, that there exists another equilibrium state $\overline{\nu}$. We will show that $\nu : = \overline{\nu} \circ \pi^{-1}$ is different from the equilibrium state $\mu$. 
		
		In fact, since $\overline{\nu} \neq \overline{\mu}$, we can find $A \in \mathcal{A}$ such that $\overline{\nu}(A) \neq \overline{\mu}(A)$. The $\sigma$-algebra $\mathcal{A}$, which coincides with the Borelians on $\Lambda$, was obtained as $\mathcal{A} = \bigcup_{n=0}^{\infty}F^{-n}(\mathcal{A}_{0})$. Thus, there exist $A_{0} \in \mathcal{A}_{0}$ and $n_{0} \in \mathbb{N}$ such that $A = F^{-n_{0}}(A_{0})$.
		
		Remember also that $\mathcal{A}_{0} = \pi^{-1}(\mathcal{A}_{\Omega})$ and therefore there exists $B \in \mathcal{A}_{\Omega}$ such that $A_{0} = \pi^{-1}(B)$. Using that $\overline{\mu}$ is invariant under $F^{-1}$ we have
		\[
		\nu(B) = \overline{\nu}(\pi^{-1}(B)) = \overline{\nu}(A_{0}) = \overline{\nu}(F^{-n_{0}}(A_{0})) = \overline{\nu}(A) \neq
		\]
		\[
		\overline{\mu}(A) = \overline{\nu}(F^{-n_{0}}(A_{0})) = \overline{\nu}(A_{0}) = \overline{\mu}(\pi^{-1}(B)) = \mu(B).
		\]
		It means that $\nu \neq \mu$ and the first part is proven.
		
		Now, assuming that the system $(F^{-1},\varphi)$ has uniqueness, let us suppose, by contradiction, that the system $(G,\phi_{|\Omega})$ admits two distinct equilibrium states $\mu$ and $\nu$. Proposition \ref{proj} guarantees that $\overline{\mu}$ and $\overline{\nu}$ are equilibrium states for the system $(F^{-1},\varphi)$. By uniqueness, we have $\overline{\mu} = \overline{\nu}$. Given $B \in \mathcal{A}_{\Omega}$, analogously to what is given above, we can consider the following
		\[
		\nu(B) = \overline{\nu}(\pi^{-1}(B)) = \overline{\nu}(A_{0}) = \overline{\nu}(F^{-n_{0}}(A_{0})) = \overline{\nu}(A) =
		\]
		\[
		\overline{\mu}(A) = \overline{\nu}(F^{-n_{0}}(A_{0})) = \overline{\nu}(A_{0}) = \overline{\mu}(\pi^{-1}(B)) = \mu(B).
		\]
		Once $B$ is arbitrary, we conclude that $\mu = \nu$. This contradiction proves the uniqueness. The Proposition is proven.
	\end{proof}

    Finally, we prove that potentials which depends on the $z$-coordinate have cohomologous which does not depend on it. We follow ideas of \cite{RS}.
    \begin{proposition}(\cite{RS} Proposition 3.1)
        Let $\phi : R_{0} \cup R_{1} \to \mathbb{R}$ be a H\"older continuous potential. There exist a H\"older continuous potential $\varphi : R_{0} \cup R_{1} \to \mathbb{R}$ satisfying:
        \begin{itemize}
            \item $\varphi$ is cohomologous to $\phi$, that is, there exists $u : R_{0} \cup R_{1} \to \mathbb{R}$ such that $\varphi = \phi + u - u \circ F^{-1}$.

            \item $\varphi$ does not depend on the $z$-coordinate.
        \end{itemize}
    \end{proposition} 	
	\begin{proof}
          They take
          \[
        u(X) = \sum_{j=0}^{\infty}(\phi \circ F^{j} - \phi \circ F^{j} \circ \pi)(X).
          \]
          and prove that $\varphi$ like above does not depend on the $z$-coordinate.
        \end{proof}
	\section{Hyperbolic Potentials versus Expanding Potentials}\label{versus}
	
	In the works \cite{AOS} and \cite{RV} the authors prove existence of equilibrium state for non-uniformly expanding maps and the so-called hyperbolic potentials. While in \cite{RV} they also prove uniqueness with a mild condition, in \cite{AOS} prove an equivalence between the so-called expanding potentials and the hyperbolic ones. It will allows us to obtain uniqueness for the horseshoe in some cases where the restriction of the potential is expanding (or hyperbolic) by using Theorem \ref{transfer}. We recall the definitions of hyperbolic and expanding potentials as it is given in \cite{AOS}.
	
	\subsection{Topological pressure}
	
	We recall the definition of relative pressure for non-compact sets by means of dynamical balls. Let $M$ be a compact metric space. Consider $f:M\to M$ and $\phi: M \to \mathbb{R}$ both continuous. Given $\delta > 0$, $n \in \mathbb{N}$ and $x \in M$, we define the 
	{\textit{dynamical ball}} $B_{\delta}(x,n)$ as 
	$$
	B_{\delta}(x,n) = \{y \in M \mid d(f^{i}(x),f^{i}(y)) < \delta, \,\, \text{for} \,\, 0 \leq i \leq n\}.
	$$
	Consider for each $N \in \mathbb{N}$, the set
	$$
	\mathcal{F}_{N} = \left\{B_{\delta}(x,n) \mid x \in M\text{ and } n \geq N\right\}.
	$$
	Given $\Lambda \subset M$, denote by $\mathcal{F}_{N}(\Lambda)$ the finite or countable families of elements   in $ \mathcal{F}_{N}$ 
	which cover $\Lambda$.
	Define for $n \in \mathbb{N}$
	$$
	S_{n}\phi(x) = \phi(x) + \phi(f(x)) + \dots + \phi(f^{n-1}(x)),
	$$
	and
	$$
	\displaystyle R_{n,\delta}\phi(x) = \sup_{y \in B_{\delta}(x,n)} S_{n}\phi(y).
	$$
	Given a $f$-invariant set $\Lambda \subset M$, not necessarily compact, define for each $\gamma \in\mathbb R$
	$$
	\displaystyle m_{f}(\phi, \Lambda, \delta,   \gamma, N) = \inf_{\mathcal{U} \in \mathcal{F}_{N}(\Lambda)} \left\{ \sum_{B_{\delta}(x,n) \in \mathcal{U}}
	e^{-\gamma n + R_{n,\delta}\phi(x)} \right\}.
	$$
	Define also
	$$
	\displaystyle m_{f}(\phi, \Lambda, \delta, \gamma) = \lim_{N \to + \infty} m_{f}(\phi, \Lambda, \delta,   \gamma,N)
	$$
	and
	$$
	P_{f}(\phi, \Lambda, \delta) = \inf \{\gamma \in\mathbb R \mid m_{f}(\phi, \Lambda, \delta, \gamma) = 0\}.
	$$
	Finally, define the {\textit{relative pressure}} of $\phi$ on $\Lambda$ as
	$$
	P_{f}(\phi,\Lambda) = \lim_{\delta \to 0} P_{f}(\phi, \Lambda, \delta).
	$$
	The {\textit{topological pressure}} of $\phi$ is, by definition, $P_{f}(\phi) = P_{f}(\phi, M)$. It satisfies
	\begin{eqnarray}
		\label{Pressures}
		P_{f}(\phi) = \sup \{P_{f}(\phi,\Lambda), P_{f}(\phi,\Lambda^{c})\},
	\end{eqnarray}
	where $\Lambda^{c}$ denotes the complement of $\Lambda$ on $M$. We refer the reader to \cite[Theorem 11.2]{P}   for the proof of~\eqref{Pressures} and for additional properties of the pressure.
	
	\subsection{Expanding measures and hyperbolic potentials}


	Given $0<\sigma<1$ and $ \varepsilon > 0$,
	we say that $n \in \mathbb{N}$ is a {\textit{hyperbolic time}} for $x\in M$ if
	\begin{itemize}
		\item there exists a neighbourhood $V_{n}(x)$ of $x$ such that $f^{n}$ maps  $ V_{n}(x)$ homeomorphically onto the ball $ B_{\varepsilon}(f^{n}(x))$;
		\item $d(f^{i}(y),f^{i}(z)) \leq \sigma^{n-i}d(f^{n}(y),f^{n}(z))$, for all $y,z \in V_{n}(x)$ and all $0 \leq i \leq n-1.$
	\end{itemize}
	We say that $x \in M$ has {\textit{positive frequency}} of $(\sigma,\varepsilon)$-hyperbolic times if
	$$
	d(x):=\limsup_{n \to \infty}\frac{1}{n}\#\{0 \leq j \leq n-1 \,\, | \,\, j \,\, \text{is a} \,\, \text{-hyperbolic time for} \,\, x\} > 0,
	$$
	and define the {\textit{expanding set}}
	$$
	H = \{x \in M \mid  d(x)>0\}.
	$$
	We say that a Borel probability  measure $\mu$ (not necessarily $f$-invariant) on $M$   is {\textit{expanding}} if $\mu(H)=1$, and that a   continuous function $\phi : M \to \mathbb{R}$ is a \emph{hyperbolic potential} if   the topological pressure $P_{f}(\phi)$ is located on $H$, i.e.
	$$P_{f}(\phi,H^{c}) < P_{f}(\phi).$$
	
	\begin{proposition}\label{pr.hypmeas}
		Let $\phi$ be  a hyperbolic potential. If $\mu$ is an ergodic probability measure such that $h_{\mu}(f) + \int \phi d\mu > P_{f}(\phi,H^{c})$, then $\mu(H)=1$.
	\end{proposition}
	
	\begin{proof}
		Since $H$ is an invariant set and $\mu$ is an ergodic probability measure, we have $\mu(H) = 0$ or $\mu(H) = 1$. Since the potential $\phi$ is hyperbolic, we get
		$$
		h_{\mu}(f) + \int \phi d\mu > P_{f}(\phi,H^{c}) \geq \sup_{\nu(H^{c})=1}
		\bigg{\{}h_{\nu}(f) + \int_{H^{c}} \phi d\nu\bigg{\}}.
		$$
		(For the second inequality, see  \cite[Theorem A2.1]{P})
		So, we cannot have $\mu(H^{c}) = 1$ and we obtain $\mu(H) = 1$, i.e. $\mu$ is an expanding measure.
	\end{proof}

	\subsection{Expanding potentials}Let $\phi:M \to \mathbb{R}$ be a continuous potential. We say that $\phi$ is an \textit{expanding potential} if the following inequality holds
	\[
	\displaystyle \sup_{\mu \in \mathcal{E}^{c}} \bigg{\{} h_{\mu}(f) + \int \phi d \mu \bigg{\}} < \sup_{\mu \in \mathcal{E}} \bigg{\{} h_{\mu}(f) + \int \phi d \mu \bigg{\}},	
	\]
	where $\mathcal{E}$ denotes the set of all expanding measures for $f$. 
	
	We observe that Proposition \ref{pr.hypmeas} shows that every hyperbolic potential is an expanding potential. Since for an expanding potential $\phi$ we can find a sequence $\mu_{n}$ of expanding measures such that we have $h_{\mu_{n}}(f) + \int \phi d \mu_{n} \to P_{f}(\phi)$, if $\phi$ is also H\"{o}lder , we also can find finitely many ergodic equilibrium states for $\phi$ which are expanding measures. In fact, all the proof is similar from the point where Proposition \ref{pr.hypmeas} is proved. We obtain the following Theorem:
	
	\begin{theorem}(\cite{AOS}[Theorem B])\label{HYPERZOOM}
		Let $f:M \to M$ be a continuous map with an expanding measure and $\phi : M \to \mathbb{R}$ a continuous potential. Then $\phi$ is a hyperbolic potential if, and only if, it is an expanding potential. In particular, if $\phi$ is hyperbolic and H\"older with finite pressure $P_{f}(\phi)$, there exist finitely many equilibrium states which are expanding measures.
	\end{theorem}

     With Theorem \ref{HYPERZOOM} we can use the following result for hyperbolic potentials.

       \begin{theorem}(\cite{RV}[Theorem 2])\label{Ramos-Viana} Let $f:M \to M$ be a local homeomorphism and let $\phi : M \to \mathbb{R}$ be a hyperbolic H\"older continuous potential. Then, there exist finitely many ergodic equilibrium states associated with $(f,\phi)$. In addition, the equilibrium state is unique if the pre-orbit $\{f^{-n}(x)\}_{n \geq 0}$ of every point $x \in H$ is dense in $M$.

       \end{theorem}
        
	We observe that in \cite{S} the zooming measures and continuous zooming potentials in the context of exponential contractions are what we call here expanding measures and expanding potentials. Hence, our Theorem \ref{HYPERZOOM} is a corollary of \cite{S}[Theorem 8.2.1]. We state this result in the following: 
	
	\begin{theorem}(\cite{S}[Theorem 8.2.1])\label{HYPER}
	Let $f:M \to M$  a continuous 
        zooming map and  the contraction $(\alpha_{n})_{n}$ satisfying $\alpha_{n}(r) \leq ar$ for some $a \in (0,1)$, every $n \in \mathbb{N}$ and every $r \in [0,+\infty)$ (Lipschitz contractions, for example)  and $\phi : M \to \mathbb{R}$ a continuous potential. Then $\phi$ is a hyperbolic potential if, and only if, it is a zooming potential. In particular, if $\phi$ is hyperbolic and H\"older with finite pressure $P_{f}(\phi)$, there exist finitely many equilibrium states which are zooming measures.
        \end{theorem}

	\section{Proof of Theorem \ref{hyperbolic}}\label{Rectangles}
	\vspace{0.50cm}
	
	In the work \cite{RS} the authors prove that the map $G$, the projection of $F^{-1}$, is a non-uniformly expanding map and satisfy our assumption of hyperbolic times in section \ref{versus}. We can now prove our Theorem \ref{hyperbolic}.
	
	\begin{proposition}\label{projectiveexpanding}
		If a potential $\phi:\Lambda \longrightarrow \mathbb{R}$ is projective hyperbolic, then its restriction $\varphi=\phi_{|\Omega}$ is an expanding potential (or hyperbolic) with respect to $G$. 
	\end{proposition}
	\begin{proof}
		Once every potential with low variation is hyperbolic (or expanding), we have that the null potential is hyperbolic and then there exists a measure of maximal entropy $\mu \in \mathcal{E}$, that is, $h(G) = h_{\mu}(G)$. Since $\varphi$ is a potential such that $\varphi(x) \geq u(x) - u(G(x))$ for some measurable function $u:\Omega \to \mathbb{R}$ and every $x \in H$ we have that  $\int \phi d\eta \geq 0$ for every $\eta \in \mathcal{E}$. It implies that 
		\[
		\sup_{\nu \in \mathcal{E}^{c}}\bigg{\{}h_{\nu}(G) + \int \phi d\nu\bigg{\}} \leq \sup_{\nu \in \mathcal{E}^{c}}\{h_{\nu}(G)\} + \sup_{\nu \in \mathcal{E}^{c}}\bigg{\{}\int \phi d\nu\bigg{\}} < h_{\mu}(G) \leq
		\]
		\[
		h_{\mu}(G) + \int \phi d\mu \implies \sup_{\nu \in \mathcal{E}^{c}}\bigg{\{}h_{\nu}(G) + \int \phi d\nu\bigg{\}} < \sup_{\mu \in \mathcal{E}}\bigg{\{}h_{\mu}(G) + \int \phi d\mu\bigg{\}}.
		\]
		It shows that the potential $\phi$ is expanding (or hyperbolic).  		
	\end{proof}
	
	\begin{corollary}
		If $\phi : \Lambda \to \mathbb{R}$ is a projective hyperbolic potential, then there exists a unique equilibrium state for $G$ associated to the potential $\varphi = \phi_{|\Omega}$.
	\end{corollary}	
	
	\begin{proof}
		By using of the Proposition \ref{projectiveexpanding} the potential $\varphi$ is expanding and by Theorem \ref{HYPERZOOM} it is also  a hyperbolic potential. In this way, since the map $G$  is strongly topologically mixing and non-uniformly expanding, by [\cite{RV}, Theorem 2] we have that there exists a unique equilibrium state for the pair $(G,\varphi)$. 
	\end{proof}
	
        Hence, by Theorem \ref{transfer} we obtain uniqueness for the horseshoe $F$.
            
	\section{Proof of Theorem \ref{ThQ}}\label{Q}
	
	\begin{lemma}
		Let $F : \Lambda \to \Lambda$ be the horseshoe, $\eta$ the unique measure of maximal entropy for the non-uniformly expanding map $G : \Omega \to \Omega$ and $\phi : \Lambda \to \mathbb{R}$ a H\"older potential and $X \in \Lambda$ such that $\sup\phi_{|\Omega} = \phi (X) < P(\phi_{|\Omega}) - h_{\eta}(G)$. Then the restriction of $\phi_{|\Omega}$ is an expanding potential, that is, it is a hyperbolic potential.
	\end{lemma}
	\begin{proof}
		The existence and uniqueness of the measure of maximal entropy $\eta$ is given by Ramos-Viana \cite{RV} because Santana \cite{S} gives that the null potential is zooming (expanding), then hyperbolic. So, denoting by $\mathcal{E}$ the set of expanding measures, since $\sup \phi_{|\Omega} = \phi(X) < P(\phi_{|\Omega}) - h_{\eta}(G)$, we have
		\[
		\sup_{\nu \in \mathcal{E}^{c}}\Bigg{\{} h_{\nu}(G) + \int (\phi_{|\Omega} - P(\phi_{|\Omega}))  d\nu \Bigg{\}} < h_{\eta}(G) + \sup_{\nu \in \mathcal{E}^{c}}\Bigg{\{} \int (\phi_{|\Omega} - \phi(X) - h_{\eta}(G))  d\nu \Bigg{\}} = 
		\]
		\[
		h_{\eta}(G) + \sup_{\nu \in \mathcal{E}^{c}}\Bigg{\{} \int (\phi_{|\Omega} - \phi(X))  d\nu \Bigg{\}}  - h_{\eta}(G) \leq  0 = 
		\]
		\[
		P(\phi_{|\Omega} - P(\phi_{|\Omega})) = \sup_{\nu \in \mathcal{E}}\Bigg{\{} h_{\nu}(G) + \int (\phi_{|\Omega} - P(\phi_{|\Omega}))  d\nu \Bigg{\}}.
		\]
		Hence, $\phi_{|\Omega} - P(\phi_{|\Omega})$ is expanding and so is $\phi_{|\Omega}$. Then, $\phi_{|\Omega}$ is hyperbolic. Ramos-Viana gives a unique equilibrium state because $G$ is topologically exact and the pre-images of points in the expanding set are dense.
	\end{proof}
	\begin{lemma}\label{cohomologous}
		Let $F : \Lambda \to \Lambda$ be the horseshoe, $\eta$ the unique measure of maximal entropy for the non-uniformly expanding map $G : \Omega \to \Omega$ and a H\"older potential such that it holds that $\sup \phi_{|\Omega} = 0 = \phi_{|\Omega}(Q)  < P(\phi_{|\Omega})$. Then, there exists a H\"older potential $\varphi : M \to \mathbb{R}$ such that $\sup \varphi_{|\Omega}  < P(\varphi_{|\Omega}) - h_{\eta}(G)$ and $\varphi_{|\Omega}$ is cohomologous to $\phi_{|\Omega}$.
	\end{lemma}
	\begin{proof}
		Once $Q$ is a fixed point, we have that
		\[
		\sup \phi_{|\Omega} = 0 = \phi(Q) = h_{\delta_{Q}}(G) + \int \phi_{|\Omega} d\delta_{Q} \leq P(\phi_{|\Omega}) \implies P(\phi_{|\Omega}) = h_{\eta}(G).
		\]
		Define for $t \in \mathbb{R}$ and $u = -t\phi$
		\[
		\varphi(x) = (1-t)\phi(x) + t \phi(F^{-1}(x)) = \phi(x) + (-t\phi)(x) - (-t\phi(F^{-1}(x))) = \phi(x) + u(x) - u(F^{-1}(x)).
		\]
		We claim that there exists $t \in \mathbb{R}$ such that $\sup \varphi_{|\Omega}  < P(\varphi_{|\Omega}) - h_{\eta}(G)$. Since  it holds that $\sup \phi_{|\Omega} = 0 = \phi_{|\Omega}(Q)  < P(\phi_{|\Omega}) = h_{\eta}(G)$, we can take $0 < t < 1 \approx s$ such that it holds that $h_{\eta}(G) < sP(\phi_{|\Omega}), s < 1 + (t-1) \sup  \phi_{|\Omega}/P(\phi_{|\Omega}) $ and for every $x \in \Omega$
		\[
		t\phi(F^{-1}(x)) = t\phi(G(x))\leq t\sup \phi_{|\Omega}\leq 0 < sP(\phi_{|\Omega}) - h_{\eta}(G),  (1-t) \sup \phi_{|\Omega}  < (1-s)P(\phi_{|\Omega}).
		\]
		So,
		\[
		\varphi_{|\Omega}(x) = (1-t)\phi_{|\Omega}(x) + t \phi_{|\Omega}(G(x)) < (1-s)P(\phi_{|\Omega}) + sP(\phi_{|\Omega}) - h_{\eta}(G) = 
		\]
		\[
		P(\phi_{|\Omega}) - h_{\eta}(G) = P(\varphi_{|\Omega}) - h_{\eta}(G),
		\]
		because $\phi$ and $\varphi$ are cohomologous and have same pressures. Then $\sup \varphi_{|\Omega}  < P(\varphi_{|\Omega}) - h_{\eta}(G)$ as we wished.
	\end{proof}
	By Lemma 1 $\varphi_{|\Omega}$ is hyperbolic or expanding and it means that it has uniqueness of equilibrium state. Once by Lemma 2 $\phi_{|\Omega}$ is cohomologous to $\varphi_{|\Omega}$, $\phi_{|\Omega}$ also has uniqueness of equilibrium state.

    \begin{remark}
      We observe that if $\phi_{|\Omega}(Q) \neq 0$ we can take the potential $\varphi = \phi - \phi_{|\Omega}(Q)$ and the result follows.
    \end{remark}
    
	\begin{lemma}
		We have either every H\"older potential $\varphi$ such that $\sup \varphi_{|\Omega} < P(\varphi_{|\Omega})$ has uniqueness of equilibrium state or there exists a H\"older potential $\phi$ such that $\phi_{|\Omega}(Q)< \sup \phi_{|\Omega} < P(\phi_{|\Omega})$ and $\phi$ does not have uniqueness of equilibrium states.
	\end{lemma}
	\begin{proof}
		Let us suppose that not every H\"older potential $\varphi$ like that has uniqueness of equilibrium state. We can suppose that $\varphi_{|\Omega} \geq 0$. Also, $\varphi_{|\Omega}$ does not have uniqueness. By Lemma \ref{cohomologous}, we have either $\sup \varphi_{|\Omega} \geq P(\varphi_{|\Omega})$ or $\varphi_{|\Omega}(Q) < \sup \varphi_{|\Omega}$. Once $\sup \varphi_{|\Omega} < P(\varphi_{|\Omega})$ by hypothesis, we then have $\varphi_{|\Omega}(Q) < \sup \varphi_{|\Omega}$
		
		Since $\sup \varphi_{|\Omega} < P(\varphi_{|\Omega})$, denote by $\text{supp} \mu$ the support of the measure $\mu$. Consider the following union
		\[
		X : = \bigcup_{\mu \,\, \text{invariant} \neq \delta_{Q}} \text{supp} \mu.
		\]
		Take $t < 0$ such that $(1 + td(Q,X))\varphi(Q) < \sup\varphi_{|X}$ and define
		\[
		\phi(x) := (1 + td(x,X))\varphi(x) \implies \phi_{|\Omega}(Q) = (1 + td(Q,X))\varphi_{|\Omega}(Q)< \sup\varphi_{|X} \leq \sup\varphi_{|\Omega} < P(\varphi_{|\Omega}).
		\]
		We used that $d(x,X) = 0$ for every $x \in X$ and $\phi_{|X} = \varphi_{|X}$. Also, we have that for $\mu \neq \delta_{Q}$ an invariant measure 
		\[
		\int\phi_{|\Omega} d\mu = \int_{\text{supp}\mu}\phi_{|\Omega} d\mu = \int_{\text{supp}\mu} (1 + d(x,X))\varphi_{|\Omega}(x) d\mu(x) = \int\varphi_{|\Omega} d\mu \implies P(\varphi_{|\Omega}) =  P(\phi_{|\Omega}),
		\]
		because
		\[
		h_{\delta_{Q}}(G) + \int \phi_{|\Omega} d\delta_{Q} = \phi_{|\Omega}(Q) < P(\varphi_{|\Omega}) = \sup_{\mu}\Bigg{\{}h_{\mu}(G) + \int \varphi_{|\Omega} d\mu\Bigg{\}} = 
		\]
		\[
		\sup_{\mu}\Bigg{\{}h_{\mu}(G) + \int \phi_{|\Omega} d\mu\Bigg{\}} = P(\phi_{|\Omega}).
		\]
		Hence, $\phi_{|\Omega}(Q) < P(\varphi_{|\Omega}) = P(\phi_{|\Omega})$. Also, once $\phi_{|X} \equiv \varphi_{|X}$, we have $\phi_{|\Omega}(Q) < \sup \varphi_{|X} = \sup \phi_{|X} \leq \sup \phi_{|\Omega}$.
		It means that $\mu$ is an equilibrium state for $\phi$ and since $\varphi$ has at least $3$ equilibrium states, $\phi$ does not have uniqueness of equilibrium state. Moreover, $\phi$ is H\"older because $\varphi$ and $d(x,X)$ are H\"older.
	\end{proof}
	It means that among the H\"older potentials $\varphi$ such that $\sup \varphi_{|\Omega} < P(\Omega)$ either all of them has uniqueness or the condition $\varphi_{|\Omega}(Q) =  \sup \varphi_{|\Omega}$ is needed to have uniqueness.

	\section{Examples}

	\subsection{Examples of Admissible Potentials}\label{Example}
	
	In this section it will be given examples of  admissible potentials with respect to the horseshoe map $F$.
	
	Firstly, we present some preliminary results and definitions used to get the conditions  $(C_{1})$ and $(C_{2})$.
	
	Let $n$ be positive integer and denote the words
	\[
	b_{n}=\underbrace{0\dots01}_{n\,\textrm{times}}\!\!\!.
	\]
	Given a sequence $w\in [1]$ and a positive integer $i$, we denote its \textit{i-segment} by
	\[
	[w]_{i}: = w_{0}w_{1} \dots w_{i-1}.
	\]
	If $\sigma^{i}(w)\in [1]$, so the \textit{i-segment} of $w$ is the concatenation of words $b_{n}$, that is,
	\[
	1b_{n_{1}}b_{n_{2}} \dots b_{n_{r}}.
	\]
	We denote the cardinally of the words $b_{n}$ in the $i$-segment of $w$ by $a(w,i,n)$. Note that if $w\in\Sigma_{\alpha}$, then it is a concatenation of infinite words of the type $b_{n}$.
	\begin{lemma}\label{low}
		Let $\tau<\alpha$ and $N= \lfloor \frac{1}{\alpha-\tau} \rfloor$ be fixed constants. Then 
		\begin{equation}\label{5}
			1 + \sum^{N-1}_{k=1} a(w,i,k)\geq\tau(i+1),
		\end{equation}
		for $w\in \Sigma_{i}$.
	\end{lemma}
	\begin{proof}
		By contradiction, assume that
		\begin{equation}\label{1}
			1 + \sum^{N-1}_{k=1} a(w,i,k)<\tau(i+1).
		\end{equation}
		Note that the number of digits 1's in $[w]_{i}$ is given by
		\[
		1 + \sum^{s}_{k=1} a(w,i,k),
		\]
		where $s$ is the largest number $n$ such that $b_{n}$ is in  $[w]_{i}$, and it is possible that some terms of the sum are null. Since $\rho(w)=i$ we have
		\begin{equation}\label{3}
			1 + \sum^{s}_{k=1} a(w,i,k) > \alpha(i+1).
		\end{equation}
		On the other hand
		\begin{equation}\label{2}
			1 + \sum^{s}_{k=1}(k+1) a(w,i,k) = (i+1).
		\end{equation}
		Then
		\[
		a(w,i,N) + a(w,i,N+1) + ... + a(w,i,s)
		>(\alpha-\tau)(i+1),
		\]
		we conclude that
		\begin{eqnarray}
			\sum^{s}_{k= N}(k+1)a(w,i,k) & \geq & (N+1) \left( \sum^{s}_{k= N}a(w,i,k) \right) \nonumber \\
			&>& (N+1)(\alpha-\tau)(i+1) \nonumber \\
			&>& (i+1), \nonumber
		\end{eqnarray}
		which contradicts the expression \eqref{2}.
	\end{proof}
	A way to understand the Lemma \ref{low} is that, if $i$ is a $\alpha$-return time of a sequence $w$ to [1], then the words $b_{1},\dots,b_{N}$ represent at least a part $\tau i$ of $[w]_{i}$.
	
	Given a point $X=(x,y,z)$ in $\Lambda$, such that $h(X)=w$, we use the follow notation
	\[
	\Phi_{[w]_{n}} (y) = f_{w_{n-1}}\circ \dots \circ f_{w_{0}}(y),
	\]
	and also the inverse case
	\[
	\Phi_{[w]_{n}^{-}} (y) = f_{w_{-1}}\circ \dots \circ f_{w_{-n}}(y).
	\]
	The next result gives us non-uniform contraction
	\begin{lemma}\label{l.LOR}[\cite{LOR}, Lemma 3.1]
		Let $w\in \Sigma_{11}^{+}$ be sequence with infinitely many 1's. Assume $w_{0}=1$.
		Let $n_{0},n_{1},\dots$ be the successive positions of the symbol 1 in w. Then, there exist a sequence
		of positive real numbers $(\delta_{j})_{j\geq0}$ and a positive real number $C$ such that the following holds.
		\begin{itemize}
			\item[(i)] C depends only on $n_{0}$.
			\item[(ii)] Each $\delta_{j}$ depends only on $n_{i}$, for  $i\leq j$ and belongs to the $[0,\sigma]$.
			\item[(iii)] For every $i>0$ and for every $y\in[0,1]$,
			\[
			|\Phi^{'}_{[w]_{n_{i}}}(y)|\leq C_{n_{0}}\prod^{i-1}_{j=1}\theta_{j}
			\]
			where $\theta_{j}=\frac{1-\frac{\delta_{j}}{\sigma}}{1-\delta_{j}}$.
		\end{itemize}
	\end{lemma}
	\begin{remark}\label{contracSIF}
		The Lemma \ref{l.LOR} is obtained by using the contraction provided by the block of the form 0...01. Since the factors of the product in $(iii)$ are all strictly less than 1, if for a block $b_{n}$ it appears $k$ times in the segment $[w]_{n_{i}}$, we have that
		\[
		|\Phi^{'}_{[w]_{n_{i}}}(y)|\leq C_{n_{0}} \theta^{k},
		\]
		where $\theta$ depends only of $b_{n}$.
	\end{remark}
	Given points $X$ and $Y$ in $\Lambda$, we use the follow norm
	\begin{equation}\label{6}
		||X-Y|| = ||X-Y||_{s} + ||X-Y||_{c} + ||X-Y||_{u},
	\end{equation}
	where $||X-Y||_{\ast}$ is the distance with respect to direction $\ast$ for $\ast\in \{s,c,u\}$.
	
	By the hyperbolic behavior of $F$ on the stable and unstable directions and using the distance  \eqref{6}, given two points $X, Y \in \Lambda$ such that $h(X),h(Y)\in [w_{-n+1}\dots w_{n-1}]$ we have
	\begin{equation}\label{stableunstable}
		||X-Y||_{s}\leq \lambda^{n}   \    \     \      \       \      \      \     \   \textrm{e}    \    \     \      \       \      \      \     \  ||X-Y||_{u}\leq \beta^{-n}.
	\end{equation}
	Now we are able to start the result of this section.
	\begin{proposition}\label{Example Prop.}
		Let $c_{0}\in (5/6,1)$ be a constant and $\mathcal{Q}_{c_{0}}=\{X=(x,y,z) \in \Lambda \ ; \  z\leq c_{0}\}$. If $\phi:\Lambda\rightarrow \mathbb{R}$ is a H\"{o}lder  continuous potential constant on $\mathcal{Q}_{c_{0}}$ with $\sup \phi = \phi(Q)$, then there exists a non-empty open interval $I\subset \mathbb{R}$ such that for every $t\in I$ the potential $\phi_{t}:=t\phi$ is admissible. 
	\end{proposition}
	\begin{proof}
		First, note that given $5/6<c_{0}<1$, we take $m=m(c_{0})$ by 
		\[
		m=\max\{k \, ; \ \beta_{0}^{k}\beta_{1}(c_{0}-5/6)> 1/6\}.
		\]
		So, for every $X\in\Lambda \cap R_{1} - \mathcal{Q}_{c_{0}}$ we have $h(X)\in [1\underbrace{0\dots0}_{m}1]$.
		Now, consider the potential $\varphi = \phi\circ h^{-1}$ and its induced $\varphi_{\rho}$ with respect to the inducing scheme of section \ref{Construction of I.E} . Given sequences $w$ and $v$ in $\Sigma_{11}$ such that $\Pi^{-1}(w)$ and $\Pi^{-1}(w)$ are in the cylinder $[D^{l_{-n+1}}_{i_{-n+1}}\dots D^{l_{0}}_{i_{0}}\dots D^{l_{n-1}}_{i_{n-1}}]$ of spaces $S^{\mathbb{Z}}$, so $w$ and  $v$ belong to
		\[
		D^{l_{0}}_{i_{0}}=[\overbrace{\overbrace{10\dots 01}^{n_{1}}\dots 10\dots 01}^{n_{r}}],
		\]
		where the numbers $n_{1},\dots, n_{r}$ are the positions of the digit 1.
		Let  $X_{0}=h^{-1}(w)$ and $Y_{0}=h^{-1}(v)$ be points in $\Lambda$ and $n_{r_{k}}$ the $k$th return of points $X_{0}$ and $Y_{0}$ to the set  $\Lambda \cap (R_{1} - \mathcal{Q}_{c_{0}})$, then we have  \begin{eqnarray}\label{somaHolder}
			|\varphi_{\rho}(w)-\varphi_{\rho}(v)| & \leq & \sum_{i=0}^{l_{0}-1} |\phi(F^{i}(X_{0}))-\phi(F^{i}(Y_{0}))|  \\
			&=& C\sum_{k=1}^{\bar{r}} ||F^{n_{r_{k}}}(X_{0}))-F^{n_{r_{k}}}(Y_{0}))||^{\xi}, \nonumber
		\end{eqnarray}
		where $\bar{r}$ is number of return of the points $X_{0}$ and $Y_{0}$ to the set $\Lambda \cap (R_{1} - \mathcal{Q}_{c_{0}})$.
		Note that the terms  of the sum \eqref{somaHolder} related to the points in $\mathcal{Q}_{c_{0}}$, are equal to zero, since $\phi$ is constant in this set. Moreover, by  \eqref{stableunstable}, it is enough to analyze the central direction.
		By using Lemma \ref{low}, we have that in each segment $[w]_{s_{n}}$, where  $s_{n}=\sum_{k=1}^{n-1} l_{-k}$, the words  $b_{1},b_{2},\dots,b_{N}$ appear at least with frequency $\tau s_{n}$. Then, by
		Lemma \ref{l.LOR} and Remark \ref{contracSIF}, we conclude
		\begin{eqnarray}\label{IV}
			|\Phi^{'}_{[w]_{s_{n}}^{-}}(y)|\leq
			C_{N} \theta^{n},
		\end{eqnarray}
		for $y\in [0,1]$ and constants  $C_{N}>0$ and $\theta\in (0,1)$ that depends only on the words $b_{1},b_{2},\dots,b_{N}$. Therefore, let $X^{\star}$ and $Y^{\star}$ be points in $\Lambda$  such that  $F^{s_{n}}(X^{\star})=X_{0}$ and $F^{s_{n}}(Y^{\star})=Y_{0}$, by the Lemma \ref{l.LOR} and the inequality  \eqref{IV} we have
		\begin{eqnarray}\label{est.1}
			||F^{n_{r_{k}}}(X_{0})-F^{n_{r_{k}}}(Y_{0})||_{c} & = & ||F^{s_{n}+n_{r_{k}}}(X^{\star}) - F^{s_{n}+n_{r_{k}}}(Y^{\star})||_{c}  \nonumber \\
			& \leq & |\Phi_{[\sigma^{n_{r_{k}}}(w)]_{s_{n}+n_{r_{k}}}^{-}}^{'}|\\
			& \leq & C_{N}\theta^{n} \, \prod_{j=1}^{r_{k}} \theta_{n_{j}} \nonumber \\
			& \leq & C_{N}\theta^{n} \, \gamma^{k} , \nonumber
		\end{eqnarray}
		where $1\leq k \leq \bar{r}$ and the constant $\gamma\in(0,1)$ depends only on word $b_{m}$. Hence, by using \eqref{est.1} in the expression \eqref{somaHolder}, we have
		\[
		|\varphi_{\rho}(w)-\varphi_{\rho}(v)|\leq \tilde{C}\Theta^{n},
		\]
		where $\Theta=\max\{\lambda,\beta, \theta\}$ and $\tilde{C}=C.C_{N}.\Upsilon$, for $\Upsilon=\sum_{k=0}^{\infty}\gamma^{k}$. Therefore, the potential $\phi$ satisfies the condition $(C_{1})$.
		
		Consider the family of the potentials  $\phi_{t}:\Lambda\rightarrow \mathbb{R}$ given by
		\[
		\phi_{t}(X)=t \phi(X),
		\]
		where $t\in\mathbb{R}$. By using an analogous argument 
		we can show that $\phi_{t}$ satisfies the condition $(C_{1})$ for every $t\in \mathbb{R}$. We denote the topological pressure of $\phi_{t}$ by $P(t)$. Note that the function  $t\mapsto P(t)$ is convex, so it is also a continuous function on $t\in\mathbb{R}$, since $h_{top}(F)>0$. Consider the set
		\[
		\varDelta=\{\vartheta ; \ t\in [0,\vartheta) \, , \, P(t)> t\phi(Q)\}.
		\]
		It is a non-empty set, since  $P_{top}(0)>0$.
		The supremum  of $\Delta$ is given by
		\[
		t_{1}=\sup_{\stackrel{\mu\in\mathcal{M}_{F}(\Lambda)}{\mu\neq \delta_{Q}}} \left\lbrace  \frac{h_{\mu}(F)}{\phi(Q)-\int \phi d \mu}\right\rbrace.
		\]
		Since the maximal entropy measure of $F$ is not the Dirac measure $\delta_{Q}$, we have
		\[
		t_{1}>\frac{h_{top}(F)}{\phi(Q)-\int \phi \, d\mu_{max}}.
		\] 
		By hypothesis on the potential $\phi$, its variation is given by
		
		\[
		Var(\phi_{t})=t\phi(Q)-t\inf\phi.
		\]
		Then, we take
		\[
		t_{0}=\sup \left\lbrace t \, ; \ Var(\phi_{t})<\frac{h_{top}(F)}{2} \right\rbrace.
		\]
		Since 
		\[
		t_{1}>\frac{h_{top}(F)}{\phi(Q)-\inf\phi },
		\]
		the interval $I=(t_{0}\, ,\, t_{1})$ is non-empty. Note that $Var(\phi_{t})\geq \frac{h_{top}(F)}{2}$, that is, $\phi_{t}$ does not have small variation.
		
		Moreover, by construction, for every $t\in (t_{0}\, , \, t_{1})$ the potential $\phi_{t}$ satisfies the condition $(C_{2})$ for $n=1$. 
		
		Therefore, for every $t\in (t_{0}\, , \, t_{1})$ the potential $\phi_{t}$ is an admissible potential.
	\end{proof}
	
	\subsection{Examples of Projective Hyperbolic Potentials}\label{Example Hyperbolic}

	As examples of projective hyperbolic potentials, we have the null one and the set of coboundaries. We will give other examples in this section. We need to find   potentials $\phi$ such that
	\begin{itemize}
		\item[$(D_{1})$] The restriction $\varphi:=\phi_{|\Omega}$ is H\"{o}lder and there exists a measurable function $u:\Omega \to \mathbb{R}$ such that  $\varphi \geq u - u \circ G$.
		\item[$(D_{2})$] If $\mathcal{E}$ denotes the set of $G$-invariant expanding measures on $\Omega$, we have that
		\[
		\sup_{\nu \in \mathcal{E}^{c}}\bigg{\{}\int \varphi d\nu\bigg{\}} < h(G) - \sup_{\nu \in \mathcal{E}^{c}}\{h_{\nu}(G)\}.
		\]  
	\end{itemize}
	
	\begin{example}\label{big}
		Let $u$ be a H\"{o}lder  function such that $\sup u - \inf u > h(G)$. Let $v$ be a H\"{o}lder  function such that $ v > u$ and $\sup (v - u \circ G)  < h(G) - \sup_{\nu \in \mathcal{E}^{c}}\{h_{\nu}(G)\}$. It implies that
		\[
		u - u \circ G < v - u \circ G \leq \sup (v - u \circ G)  < h(G) - \sup_{\nu \in \mathcal{E}^{c}}\{h_{\nu}(G)\}.
		\]
		By taking $\phi  = v - u \circ G$, we get a potential which is projective admissible.
	\end{example}
	
	In order to obtain projective hyperbolic potentials $\phi$ with big variation, that is, such that $\sup \phi - \inf \phi > h(F)$, we use the fact that $h(F) = h(G)$ (see \cite{RS}) and Example \ref{big}. We take a coboundary $u - u \circ F$ with big variation and, since the H\"{o}lder  functions are dense in the space of continuous functions, we can take $\phi$ such that $\phi \geq u - u \circ G$ and they are close enough to have the integrals close enough. Since the integrals of $u - u \circ G$ are all zero, we can take $\phi$ such that $\sup_{\nu \in \mathcal{E}^{c}}\{\int \phi d\nu\} < h(G) - \sup_{\nu \in \mathcal{E}^{c}}\{h_{\nu}(f)\}$. This finishes the proof.
	
	\bibliographystyle{amsplain}

\begin{thebibliography}{10}
		
		\bibitem{AMO} ARBIETO, A., MATHEUS, C., OLIVEIRA, K., {\it Equilibrium States for Random Non-Uniformly Expanding Maps}, \textbf{Nonlinearity}, 17, (2004), 581-593. 
		
		\bibitem {AOS}  ALVES, J. F., K. OLIVEIRA, K, SANTANA, E., \textit{Equilibrium States for Hyperbolic Potentials via Inducing Schemes}, \textbf{Nonlinearity}, 37, (2024), 095030.  
		
		
		
		\bibitem {AP} ARBIETO, A., PRUDENTE, L., \textit{Uniqueness of equilibrium states for some partially hyperbolic horseshoes}, \textbf{Discrete and Continuous Dynamical Systems},
		\textit{32} (2012), pp. 27-40.
		
		\bibitem{B} BOWEN, R., \textit{Entropy for group endomorphisms and homogeneous spaces}, \textbf{Trans. Amer. Math. Soc.} \textit{153} (1971), pp. 401-414.	
		
		\bibitem {Bowen} BOWEN, R. \textit{Equilibrium states and the ergodic theory of Anosov diffeomorphisms}, \textbf{Lecture Notes in Mathematics} , vol. \textit{470}, Springer-Verlag (1975).
		
		
		\bibitem{BK}  BRUIN, H., KELLER, G., \textit{Equilibrium states for S-unimodal maps}, \textbf{Ergodic Theory and Dynamical Systems}, 18, (1998), pp. 765-789.
		
		\bibitem{BS} BUZZI, J., SARIG, O., {\it Uniqueness of Equilibrium Measures for Countable Markov Shifts and Multidimensional Piecewise Expanding Maps}, \textbf{Ergodic Theory and Dynamamical Systems}, 23, (2003), 1383-1400.
		
		
		\bibitem {CN} CASTRO, A., NASCIMENTO, T., \textit{Statistical properties of the maximal entropy measure for partially hyperbolic attractors
		}, \textbf{Ergodic Theory and Dynamical Systems},(2016), pp 1-42.
		
		
		\bibitem {CFT} CLIMENHAGA, V., FISHER, T., THOMPSON, J., \textit{Unique equilibrium states for Bonatti-Viana diffeomorphisms}, \textbf{Nonlinearity} 31(6), (2018), pp 2532-2570.
		
		\bibitem {CT} CRISOSTOMO, J., TAHZIBI, T., \textit{Equilibrium states for partially hyperbolic diffeomorphisms with hyperbolic linear part}, \textbf{Nonlinearity}, \textbf{32}(2), (2019).
		
		
		\bibitem{DF} D\'IAZ, L., FISHER, T., \textit{Symbolic extensions for partially hyperbolic diffeomorphisms}, \textbf{Discrete and Continuous Dynamical Systems},
		\textit{29} (2011), pp. 1419-1441.
		
		
		
		\bibitem {DHRS} D\'IAZ, L	., HORITA, V., RIOS, I., SAMBARINO, M., \textit{Destroying horseshoes via heterodimensional
			cycles: generating bifurcations inside homoclinic classes}, \textbf{Ergodic Theory and Dynamical
			Systems},
		\textit{29}, (2009), pp. 433-474.
		
		
		\bibitem{DU} DENKER. M., URBANSKI, M., {\it Ergodic Theory of Equilibrium States for Rational Maps}, \textbf{Nonlinearity}, 4, (1991), 103-134.
		
		
		\bibitem {God} GODIN, \textit{On the central limit theorem for stationary processes},
		\textbf{Akademi Nauk SSSR},
		\textit{188:4} (1969), pp. 739-741.
		
		\bibitem{IT1}  IOMMI, G., TODD, M., \textit{Thermodynamic Formalism for Interval Maps: Inducing Schemes}, \textbf{Dynamical Systems}, 28, 3, (2013), pp. 354-380. 
		
		\bibitem{IT2} IOMMI, G., TODD, M., \textit{Natural Equilibrium States for Multimodal Maps}, \textbf{Communications in Mathematical Physics}, 300, (2010), 65-94. 
		
		\bibitem{LW} LEDRAPPIER, F., WALTERS, P. \textit{A relativised variational principle for continuous transformations}, \textbf{J. Lond. Math. Soc.} (2) \textbf{16} (1977), pp. 568-579.
		
		\bibitem {LOR} LEPLAIDEUR, R., OLIVEIRA, K., RIOS, I., \textit{Equilibrium States for partially hyperbolic horseshoes},
		\textbf{Ergodic Theory and Dynamical Systems},
		\textit{31} (2011), pp 179-195.
		
		
		
		\bibitem {LRL2011} LI, H., RIVERA-LETELIER, J.\textit{Equilibrium States of Weakly Hyperbolic One-Dimensional Maps for H\"{o}lder Potentials},
		\textbf{Commun. Math. Phys.},
		\textit{31} (2011), pp 328-397.
		
		\bibitem {LRL2014} LI, H., RIVERA-LETELIER, J.\textit{Equilibrium states of interval maps for hyperbolic potentials},
		\textbf{Nonlinearity},
		\textit{27} (2014).
		
		
		\bibitem{BookKV} OLIVEIRA, K., VIANA M., \textit{Foundations of ergodic theory (Cambridge Studies in Advanced Mathematics,
			151)  }, Cambridge University Press, Cambridge, 2016.
		
		\bibitem{O} OLIVEIRA, K., {\it Equilibrium States for Non-Uniformly Expanding Maps}, \textbf{Ergodic Theory and Dynamical Systems}, 23, 06, (2003), 1891-1905.
		
		
		\bibitem {OV}	OLIVEIRA, K., VIANA M., \textit{Thermodynamical formalism for robust classes of potentials and nonuniformly hyperbolic maps}, \textbf{Ergodic Theory and Dynamical Systems}, \textit{28} (2008), pp 501-533.
		
		\bibitem {P} PESIN, YA., \textit{Dimension Theory in Dynamical Systems: Contemporary Views and Applications}, The University of Chicago Press, (1997).
		
		\bibitem {PS} PESIN, Y., SENTI, S., \textit{Equilibrium measures for maps with
			inducing schemes},
		\textbf{J. Mod. Dyn.},
		2(3) (2008), pp 397-430.
		
		\bibitem {PSZlifted} PESIN, Y., SENTI, S., ZHANG, K.,\textit{Lifting measures to inducing schemess},
		\textbf{Ergodic Theory and Dynamical Systems},
		\textit{28} (2008), pp 553-574.
		
		\bibitem {PSZ} PESIN, Y., SENTI, S., ZHANG, K., \textit{Thermodynamics of towers of hyperbolic type} \textbf{Trans. Amer. Math. Soc.}\textit{368} (2016), pp 8519-8552.
		
		\bibitem{PV} PINHEIRO, V., VARANDAS, P., {\it Thermodynamic Formalism for Expanding Measures}, preprint, arXiv:2202.05019v2
		
		
		\bibitem {RS2} RAMOS, V., SIQUEIRA, J., \textit{On Equilibrium States for Partially Hyperbolic Horseshoes: Uniqueness and Statistical Properties}, \textbf{Bulletin of the Brazilian Mathematical Society}, New Series, (2017), pp 1-29.
		
		\bibitem {RV} RAMOS, V. , VIANA, M., \textit{Equilibrium states for hyperbolic potentials}, \textbf{Nonlinearity}, \textit{30} (2017), pp 825.
		
		
		\bibitem {RS} RIOS, I., SIQUEIRA, J., \textit{On equilibrium state for partially hyperbolic horseshoes}, \textbf{Ergod. Theor. Dynam. Syst.} \textit{38}, (2018), pp 301-35.
		
		\bibitem {S} SANTANA. E., \textit{Equilibrium States for Open Zooming Systems}, preprint arXiv:2010.08143.
		
		
		\bibitem {Sar99} SARIG, O., \textit{Thermodynamic formalism for countable Markov shifts},
		\textbf{Ergodic Theory Dynam. Systems},
		\textit{19} (1999), pp 1565-1593.
		
		\bibitem {Sar03} SARIG, O., \textit{Characterization of the existence of Gibbs measure for countable Markov shift},
		\textbf{Proc of AMS},
		\textit{131:6} (2003), pp 1751-1758.
		
		\bibitem {Sar15} SARIG, O., \textit{Thermodynamic formalism for countable Markov shifts},
		\textbf{Proc. of Symposia in Pure Math},
		\textit{89} (2015), pp 81-117 1751-1758.
		
		\bibitem {VV} VARANDAS, P., VIANA, M. \textit{Existence, uniqueness and stability of equilibrium states for non-uniformly expanding maps}
		\textbf{Annales Inst. Henri Poincar\'e - Analyse Non-Lin\'eaire} 27 (2009), 555-593.
		
		
		\bibitem {Y} YOUNG, L-S., \textit{Statistical properties of dynamical systems with some hyperbolicity},
		\textbf{Ann. of Math.} (2),
		\textit{147(3)} (1998), pp  585-650.
		
		\bibitem {Z} ZWEIM\"ULLER, R., \textit{Invariant measures for general(ized) induced transformations},
		\textbf{Proc. Amer.Math.Soc.} ,
		\textit{133} (2005), pp 2283-2295(electronic).
		
	\end{thebibliography}

\end{document}